\documentclass[final]{elsarticle}

\usepackage{latexsym}
\usepackage{amsmath}
\usepackage{amsthm}
\usepackage{amssymb}
\usepackage{amsfonts}
\usepackage{verbatim}
\usepackage{graphicx}
\usepackage{epstopdf}
\usepackage{subfigure}
\usepackage{float}
\usepackage{multirow}
\usepackage{color}

\usepackage[notref, notcite]{showkeys}
\usepackage[colorlinks=true,urlcolor=blue,citecolor=red,
anchorcolor=black,linkcolor=blue]{hyperref}

\newtheorem{theorem}{Theorem}[section]
\newtheorem{remark}{Remark}[section]
\newtheorem{ap}{Assumption}[section]

\newtheorem{prop}{Proposition}[section]
\newtheorem{lm}{Lemma}[section]

\let \ssection=\section
\renewcommand{\section}{\setcounter{equation}{0}\ssection}
\newcommand{\<}{\langle}
\renewcommand{\>}{\rangle}
\newcommand{\nn}{\nonumber}

\begin{document}

\begin{frontmatter}

\title{Convergence Analysis of an Unconditionally Energy Stable  Linear Crank-Nicolson Scheme for the Cahn-Hilliard Equation}

 \author[address1,address2]{Lin Wang}
 \ead{wanglin@lsec.cc.ac.cn}

 \author[address1,address2]{Haijun Yu\corref{corau}}
 \ead{hyu@lsec.cc.ac.cn}

 \address[address1]{School of Mathematical Sciences,
 	University of Chinese Academy of Sciences, Beijing
 	100049, China}

 \address[address2]{NCMIS \& LSEC, Institute of
 	Computational Mathematics and Scientific/Engineering
 	Computing, Academy of Mathematics and Systems
 	Science, Beijing 100190, China}

 \cortext[corau]{Corresponding author.}

\begin{abstract}
  Efficient and unconditionally stable high order time
  marching schemes are very important but not easy to
  construct for nonlinear phase dynamics. In this paper, we
  propose and analysis an efficient stabilized linear
  Crank-Nicolson scheme for the Cahn-Hilliard equation with
  provable unconditional stability.  In this scheme the
  nonlinear bulk force are treated explicitly with two
  second-order linear stabilization terms.  The
  semi-discretized equation is a linear elliptic system
  with constant coefficients, thus robust and efficient
  solution procedures are guaranteed.  Rigorous error
  analysis show that, when the time step-size is small
  enough, the scheme is second order accurate in time with a
  prefactor controlled by some lower degree polynomial of
  $1/\varepsilon$. Here $\varepsilon$ is the interface
  thickness parameter. Numerical results are presented to
  verify the accuracy and efficiency of the scheme.
\end{abstract}


\end{frontmatter}

\section{Introduction}

In this paper, we consider numerical approximation for the
Cahn-Hilliard equation
\begin{equation}\label{eq:CH0}
\begin{cases}
\phi_{t}=-\gamma\Delta (\varepsilon \Delta \phi - \dfrac{1}{\varepsilon}f(\phi)),   & (x,t)\in \Omega \times (0,T],\\
\phi |_{t=0} =\phi_{0}(x), &x \in \Omega,
\end{cases}
\end{equation}
with Neumann boundary condition
\begin{equation}\label{eq:CH:nbc0}
\partial_{n}\phi =0, \; \partial_{n}(\varepsilon \Delta \phi - \dfrac{1}{\varepsilon}f(\phi)) =0, \quad
x\in \partial\Omega.
\end{equation}
Here $\Omega \in R^{d}, d=2,3$ is a bounded domain with a
locally Lipschitz boundary, n is the outward normal, T is a
given time, $\phi(x,t)$ is the phase-field
variable. Function $f(\phi)=F'(\phi)$, with $F(\phi)$ is a
given energy potential with two local minima, e.g. the
double well potential $F(\phi)=\frac{1}{4}(\phi^2-1)^2$. The
two minima of $F$ produces two phases, with the typical
thickness of the interface between two phases given by
$\varepsilon$. $\gamma$ is a time relaxation parameter, its
value is related to the time unit used in a physical
process.

The equation \eqref{eq:CH0} is a fourth-order partial
differential equation, which is not easy to solve using a finite
element method.  However, if we introduce a new variable
$\mu$, called chemical potential, for
$-\varepsilon \Delta \phi + \dfrac{1}{\varepsilon}f(\phi)$,
the equation \eqref{eq:CH0} can be rewritten as a system of
two second order equations
\begin{equation}\label{eq:CH}
  \begin{cases}
    \phi_{t}=\gamma\Delta \mu,   & (x,t)\in \Omega \times (0,T],\\
    \mu=-\varepsilon \Delta \phi + \dfrac{1}{\varepsilon}f(\phi),  &(x,t)\in \Omega \times (0,T],\\
    \phi |_{t=0} =\phi_{0}(x), &x \in \Omega.
  \end{cases}
\end{equation}
The corresponding Neumann boundary condition reads
\begin{equation}\label{eq:CH:nbc}
	\partial_{n}\phi =0, \; \partial_{n}\mu =0, \quad
x\in \partial\Omega.
\end{equation}

The Cahn-Hilliard equation was originally introduced by
Cahn-Hilliard~\cite{cahn_free_1958} to describe the phase
separation and coarsening phenomena in non-uniform systems
such as alloys, glasses and polymer mixtures. If the term
$\Delta \mu$ in equation \eqref{eq:CH} is replaced with
$-\mu$, one get the Allen-Cahn equation, which was
introduced by Allen and Cahn~\cite{allen_microscopic_1979}
to describe the motion of anti-phase boundaries in
crystalline solids.
The Cahn-Hilliard equation and the Allen-Cahn equation are
two widely used phase-field model. In a phase-field model,
the information of interface is encoded in a smooth phase
function $\phi$. In most parts of the domain $\Omega$, the
value of $\phi$ is close to local minima of $F$. The
interface is a thin layer of thickness $\varepsilon$
connecting regions of different local minima.  It is easy
to deal with dynamical process involving morphology changes
of interfaces using phase-field models. For
this reason, phase field models have been the subject of
many theoretical and numerical investigations (cf., for
instance, \cite{du_numerical_1991},
\cite{elliott_error_1992}, \cite{chen_spectrum_1994},
\cite{caffarelli_l_1995}, \cite{elliott_cahnhilliard_1996},
\cite{eyre_unconditionally_1998},
\cite{furihata_stable_2001}, \cite{liu_phase_2003},
\cite{feng_error_2004}, \cite{kessler_posteriori_2004},
\cite{shen_numerical_2010}, \cite{condette_spectral_2011}).

However, numerically solving the phase-field equations is not an
easy task, since the small parameter $\varepsilon$ in the Cahn-Hilliard 
equation makes the equation very stiff and
requires a high spatial and temporal grid resolution. 
To design an energy stable scheme, one should respect the physical dissipation law of the Cahn-Hilliard system. 
In fact, the Cahn-Hilliard equation is $H^{-1}$ gradient flow of the Ginzburg-Laudau energy
functional
\begin{equation}\label{eq:CH:E}
  E(\phi):=\int_{\Omega}\Big(\frac{\varepsilon}{2}|\nabla \phi|^{2} +\frac{1}{\varepsilon} F(\phi)\Big)dx
\end{equation}
More precisely, by taking the inner product of
(\ref{eq:CH}) with $\mu$, and integration in time, we immediately find the following
energy law for (\ref{eq:CH}):
\begin{equation}\label{eq:CH:Edis}
E(\phi(t))+\gamma\int_0^t\!\!\int_{\Omega}
  |\nabla\mu|^{2}dx = E(\phi_0),\; \forall\,t>0.
\end{equation}

Since the nonlinear energy $F$ is neither a convex nor a
concave function, treating it fully explicit or implicit in
a time discretization will not lead to an efficient
scheme. In fact, if the nonlinear force $f$ is treated fully
explicitly, the resulting scheme will require a very tiny
step-size to be stable(cf. for instance
\cite{shen_numerical_2010}). On the other hand, treating it
fully implicitly will lead to a nonlinear system, for which
the solution existence and uniqueness requires a restriction
on step-size as well (cf. e.g. \cite{feng_error_2004}). One
popular approach to solve this dilemma is the convex
splitting method \cite{elliott_global_1993,eyre_unconditionally_1998}, in which the convex
part of $F$ is treated implicitly and the concave part
treated explicitly. The scheme is of first order accurate
and unconditional stable. In each time step, one need solve
a nonlinear system. The solution existence and uniqueness is
guaranteed since the nonlinear system corresponds to a
convex optimization problem. The convex splitting method was
used widely, and several second order extensions were
derived in different situations
\cite{condette_spectral_2011,baskaran_energy_2013,
  chen_linear_2014,guo_h2_2016}, etc.  Another type
unconditional stable scheme is the secant-line method
proposed by \cite{du_numerical_1991}. It is also used and
extended in several other works,
e.g. \cite{furihata_stable_2001, kim_conservative_2004,
  feng_fully_2006,
  condette_spectral_2011,gomez_provably_2011,
  baskaran_energy_2013,zhang_adaptive_2013,
  benesova_implicit_2014}. Like the fully implicit method,
the usual second order convex splitting method and the
secant-type method for Cahn-Hilliard equation need a small
time step-size to guarantee the semi-discretized nonlinear
system has a unique solution (cf. for instance
\cite{du_numerical_1991, barrett_finite_1999}). To remove
the restriction on time step-size, a diffusive three-step
Crank-Nicolson scheme was introduced by \cite{guo_h2_2016}
and \cite{diegel_stability_2016} coupled with a second order
convex splitting. After time-discretization, one get a
nonlinear but unique solvable problem at each time step.

Recently, a new approach termed as invariant energy
quadratization (IEQ) was introduced to handle the nonlinear
energy. When applying to Cahn-Hilliard equation, it first
appeared in
\cite{guillen-gonzalez_linear_2013,guillen-gonzalez_second_2014}
as a Lagrange multiplier method. It then generalized by Yang
et al. and successfully extended to handle several very
complicated nonlinear phase-field models
\cite{yang_linear_2016,han_numerical_2017,yang_numerical_2017,yang_efficient_2017,yang_yu_efficient_2017}. In
the IEQ approach, a new variable which equals to the square
root of $F$ is introduced, so the energy is written into a
quadratic form in terms of the new variable. By using
semi-implicit treatments to the nonlinear equation using new variables, one get a linear and energy stable scheme. It is
straightforward to prove the unconditional stability for
both first order and second order IEQ schemes. Comparing to
the convex splitting approach, IEQ leads to well-structured
linear system which is easier to solve. The modified energy
in IEQ is an order-consistent approximation to the original
system energy. At each time step, it needs to solve a linear
system with time-varying coefficients.

Another trend of improving numerical schemes for phase-field
models focuses on algorithm efficiency. Chen and Shen, and
their coworkers
\cite{chen_applications_1998,zhu_coarsening_1999} studied
stabilized some semi-implicit Fourier-spectral methods to the
Cahn-Hilliard equation. The space variables are discretized
by using a Fourier-spectral method whose convergence rate is
exponential in contrast to the second order convergence of a
usual finite-difference method, the time variable is
discretized by using semi-implicit schemes which allow much
larger time step sizes than explicit schemes. Xu and Tang in
\cite{xu_stability_2006} introduced a different stabilized
term to build stable large time-stepping semi-implicit
methods for an epitaxial growth model. He et al
\cite{he_large_2007} proposed similar large time-stepping
methods for the Cahn-Hilliard equation, in which a
stabilized term $A(\phi^{n+1}-\phi^n)$
(resp. $A(\phi^{n+1}-2\phi^n+\phi^{n-1})$) is added to the
nonlinear bulk force for the first order (resp. second
order) scheme. Shen and Yang systematically studied
stabilization schemes to the Allen-Cahn equation and the
Cahn-Hilliard equation in mixed formulation
\cite{shen_numerical_2010}. They got first-order
unconditionally energy stable schemes and second-order
semi-implicit schemes with reasonable stability
conditions. This idea was followed up in
\cite{feng_stabilized_2013} for the stabilized
Crank-Nicolson schemes for phase field models.  In
\cite{wu_stabilized_2014} another second-order time-accurate
schemes for diffuse-interface models, which are of
Crank-Nicolson type with a new convex-concave splitting of
the energy and tumor-growth system. In above mentioned
schemes, when the nonlinear force is treated explicitly, one
can get energy stability with reasonable stabilization constant
by introducing a proper stabilized term and a suitably
truncated nonlinear $\tilde{f}(\phi)$ instead of
$f(\phi)$ such that a uniform Lipschitz
condition 
is satisfied. It is worth to mention that with no truncation
made to double-well potential $F(\phi)$, Li et al \cite{li_characterizing_2016,
  li_second_2017} proved that the energy stable can be
obtained as well, but a much larger stability constant need
be used.

Recently, we proposed two second-order unconditionally
stable linear schemes based on Crank-Nicolson method (SL-CN)
and second-order backward differentiation formula (SL-BDF2)
for the Cahn-Hilliard equation\cite{wang_two_2017}.  In both
schemes, explicit extrapolation is used for the nonlinear
force with two extra stabilization terms which consist to
the order of the schemes added to guarantee energy
dissipation.  The proposed methods have several merits: 1)
They are second order accurate; 2) They lead to linear
systems with constant coefficients after time
discretization, thus robust and efficient solution
procedures are guaranteed; 3) The stability analysis bases
on Galerkin formulation, so both finite element methods and
spectral methods can be used for spatial discretization to
conserve volume fraction and satisfy discretized energy
dissipation law.  An optimal error estimate in
$l^{\infty}(0,T;H^{-1})\cap l^{2}(0,T;H^{1})$ norm is
obtained for the SL-BDF2 scheme in last paper.  This paper
aims to give an optimal error estimate of the SL-CN scheme.

The remain part of the paper is organized as follows. In
Section 2, we present the stabilized linear semi-implicit
Crank-Nicolson scheme for the Cahn-Hilliard equation and its
unconditionally energy stability property. In Section 3, we
carry out the error estimate to derive a convergence result
that does not depend on $1/\varepsilon$ exponentially.  A
few numerical tests for a 2-dimensional square domain are
included in Section 4 to verify our theoretical results. We
end the paper with some concluding remarks in Section 5.

\section{The stabilized linear semi-implicit Crank-Nicolson scheme}

We first introduce some notations which will be used
throughout the paper. We use $\|\cdot\|_{m,p}$ to denote the
standard norm of the Sobolev space $W^{m,p}(\Omega)$. In
particular, we use $\|\cdot\|_{L^p}$ to denote the norm of
$W^{0,p}(\Omega)=L^{p}(\Omega)$; $\|\cdot\|_{{m}}$ to denote
the norm of $W^{m,2}(\Omega)=H^{m}(\Omega)$; and $\|\cdot\|$
to denote the norm of $W^{0,2}(\Omega)=L^{2}(\Omega)$.  Let
$(\cdot, \cdot)$ represent the $L^{2}$ inner product.  In
addition, define for $p\geq 0$
\[
H^{-p}(\Omega):=\left(H^{p}(\Omega)\right)^{*},\quad
H_{0}^{-p}(\Omega):=\left\{ u \in H^{-p}(\Omega)
\mid\,\<u,1\>_{p}=0 \right\},
\]
where $\<\cdot,\cdot\>_{p}$ stands for the dual product
between $H^{p}(\Omega)$ and $H^{-p}(\Omega)$. We denote
$L_{0}^{2}(\Omega):= H_{0}^{0}(\Omega)$. For
$v \in L_{0}^{2}(\Omega)$, let
$-\Delta^{-1}v:=v_{1} \in H^{1}(\Omega)\cap
L_{0}^{2}(\Omega)$, where $v_{1}$ is the solution to
\begin{equation}\nn
-\Delta v_{1}=v \ \ {\rm in}\  \Omega ,\quad \ \frac{\partial v_{1}}{\partial n}=0 \ \ {\rm on}\  \partial \Omega,
\end{equation}
and $\|v\|_{-1}:=\sqrt{(v,-\Delta^{-1}v) }$.

For any given function $\phi(t)$ of $t$, we use $\phi^n$ to
denote an approximation of $\phi(n\tau)$, where $\tau$ is
the step-size. We will frequently use the shorthand
notations: $\delta_{t}\phi^{n+1}:=\phi^{n+1}-\phi^{n}$,
$\delta_{tt}\phi^{n+1}:=\phi^{n+1}-2\phi^{n}+\phi^{n-1}$,
and $\hat{\phi}^{n+\frac{1}{2}}
:=\frac{3}{2}\phi^{n}-\frac{1}{2}\phi^{n-1}$. Following identities and inequality will be used frequently.
\begin{equation}\label{eq:ID:1}
2(h^{n+1}-h^n, h^{n+1}) = \|h^{n+1}\|^2 - \|h^n\|^2 + \|h^{n+1}-h^n\|^2,
\end{equation}
\begin{equation}\label{eq:ID:3}
(u, v) \leq \|u\|_{-1}  \| \nabla v\|, \quad \forall\   u\in L_0^2, v\in H^1.
\end{equation}

Suppose $\phi^0=\phi_0(\cdot)$ and
$\phi^1\approx \phi(\cdot,\tau)$ are given, our stabilized liner Crank-Nicolson scheme (abbr. SL-CN)
calculates $\phi^{n+1}, n=1,2,\ldots,N=T/\tau-1$
iteratively, using
\begin{gather}\label{eq:CN:1}
\frac{\phi^{n+1}-\phi^{n}}{\tau}=\gamma\Delta \mu^{n+\frac{1}{2}},\\
\label{eq:CN:2}
\mu^{n+\frac{1}{2}}=-\varepsilon \Delta \Big(\frac{\phi^{n+1}+\phi^{n}}{2} \Big)
+\frac{1}{\varepsilon}f\Big(\frac{3}{2}\phi^{n}-\frac{1}{2}\phi^{n-1}\Big)
-A\tau \Delta \delta_{t}\phi^{n+1}
+B\delta_{tt}\phi^{n+1},
\end{gather}
where $A$ and $B$ are two non-negative constants to
stabilize the scheme.

To prove energy stability of the numerical schemes, we
assume that the derivative of $f$ in equation \eqref{eq:CH}
is uniformly bounded, i.e.
\begin{equation}\label{eq:Lip}
	\max_{\phi\in\mathbf{R}} | f'(\phi) | \le L,
\end{equation}
where $L$ is a non-negative constant.
Note that, although most of the nonlinear potential, 
e.g. the double-well poential doesn't satisfy \eqref{eq:Lip}, the above assumption is reasonable since: 1) physically $\phi$ should take values in $[-1,1]$; 2) it was proved by Caffarelli and Muler 
\cite{caffarelli_l_1995} that an $L^\infty$ bound exists
for Cahn-Hilliard equation with a potential having linear growth for $|\phi| > 1$, 3) it is proved by \cite{alikakos_convergence_1994} and \cite{feng_numerical_2005} that when a proper initial condition is given, the Cahn-Hilliard equation converges to Hele-Shaw problem when $\varepsilon\rightarrow 0$. If the corresponding Hele-Shaw problem has a global (in time) classical solution, then the solution to the Cahn-Hilliard equation has a $L^\infty$ bound.

\begin{theorem}\label{cn}
	Under the condition
	\begin{equation}\label{eq:CN:ABcond}
	A\geq \dfrac{L^{2}}{16\varepsilon^{2}}\gamma, \quad  
	B\geq \dfrac{L}{2\varepsilon},
	\end{equation}
	the following energy dissipation law
	\begin{equation}\label{eq:CN:Edis}
	E_{CN}^{n+1}\leq E_{CN}^{n}-\Big(2\sqrt{\frac{A}{\gamma}}-\frac{L}{2\varepsilon}\Big)\|\delta_{t}\phi^{n+1}\|^{2}
	-\Big(\frac{B}{2}-\frac{L}{4\varepsilon}\Big)\|\delta_{tt}\phi^{n+1}\|^{2},\hspace{1cm} \forall n\geq1,
	\end{equation}
	holds for the scheme (\ref{eq:CN:1})-(\ref{eq:CN:2}), where
	\begin{equation}\label{eq:CN:E}
	E_{CN}^{n+1}=E(\phi^{n+1})
	+\Big(\frac{L}{4\varepsilon}+\frac{B}{2}\Big)\|\delta_{t}\phi^{n+1}\|^{2}.
	\end{equation}
\end{theorem}

\begin{proof}
	Pairing \eqref{eq:CN:1} with
	$\tau \mu^{n+\frac{1}{2}}$, \eqref{eq:CN:2} with
	$-\delta_t\phi^{n+1}$, and combining the results, we get
	\begin{equation}\label{cn5}
	\begin{split}
	&\frac{\varepsilon}{2}(\|\nabla \phi^{n+1}\|^{2} -
	\|\nabla \phi^{n}\|^{2})
	+\frac{1}{\varepsilon}(f\big(\hat{\phi}^{n+\frac{1}{2}}\big),\delta_t\phi^{n+1} )\\
	=&-\gamma\tau \|\nabla \mu^{n+\frac{1}{2}}\|^{2}-A\tau \|\nabla
	\delta_{t}\phi^{n+1}\|^{2}
	-B(\delta_{tt}\phi^{n+1},\delta_t\phi^{n+1}).
	\end{split}
	\end{equation}
	Pairing (\ref{eq:CN:1}) with
	$2\sqrt{\frac{A}{\gamma}}\tau\delta_t\phi^{n+1}$, then using
	Cauchy-Schwartz inequality, we get
	\begin{equation}\label{cn7}
	2\sqrt{\tfrac{A}{\gamma}}\|\delta_{t} \phi^{n+1}\|^{2}
	= -2\sqrt{A\gamma}\tau(\nabla \mu^{n+\frac12},\nabla\delta_t \phi^{n+1})
	\leq \gamma\tau \|\nabla \mu^{n+\frac12}\|^{2} +
	A\tau\|\nabla \delta_{t}\phi^{n+1}\|^{2}.
	\end{equation}
	To handle the term involving $f$,  we expand $F(\phi^{n+1})$ and $F(\phi^n)$ at
	$\hat{\phi}^{n+\frac{1}{2}}$ as
	\begin{align*}
	F(\phi^{n+1}) &= F(\hat{\phi}^{n+\frac{1}{2}})+f(\hat{\phi}^{n+\frac{1}{2}})(\phi^{n+1}-\hat{\phi}^{n+\frac{1}{2}})+\frac{1}{2}f'(\xi^{n}_{1})(\phi^{n+1}-\hat{\phi}^{n+\frac{1}{2}})^{2},\\
	F(\phi^{n}) &= F(\hat{\phi}^{n+\frac{1}{2}})+f(\hat{\phi}^{n+\frac{1}{2}})(\phi^{n}-\hat{\phi}^{n+\frac{1}{2}})+\frac{1}{2}f'(\xi^{n}_{2})(\phi^{n}-\hat{\phi}^{n+\frac{1}{2}})^{2},
	\end{align*}
	where $\xi^{n}_1$ is a number between $\phi^{n+1}$ and
    $\hat{\phi}^{n+\frac12}$, $\xi^{n}_2$ is a number
    between $\phi^n$ and $\hat{\phi}^{n+\frac12}$.  Taking
    the difference of above two equations, we have
	\begin{equation}\nonumber
	\begin{split}
	&  F(\phi^{n+1})-F(\phi^{n}) - f(\hat{\phi}^{n+\frac{1}{2}})(\phi^{n+1}-\phi^{n})\\
	={} &  \frac{1}{2}f'(\xi^{n}_{1})
	\left[(\phi^{n+1}-\hat{\phi}^{n+\frac{1}{2}})^{2} - (\phi^{n}-\hat{\phi}^{n+\frac{1}{2}})^{2}
	\right]
	- \frac{1}{2}(f'(\xi^{n}_{2})-f'(\xi^{n}_{1}))(\phi^{n}-\hat{\phi}^{n+\frac{1}{2}})^{2}\\
	={} &  
	\frac{1}{2}f'(\xi^{n}_{1})\delta_{t}\phi^{n+1}\delta_{tt}\phi^{n+1}
	- \frac{1}{8}(f'(\xi^{n}_{2})-f'(\xi^{n}_{1}))(\delta_{t}\phi^{n})^{2}\\
	\le{} & 
	\frac{L}{4}(|\delta_t \phi^{n+1}|^2 + |\delta_{tt}\phi^{n+1}|^2)
	+ \frac{L}{4}|\delta_t\phi^n|^2.
	\end{split}
	\end{equation}
	Multiplying the above equation with $\dfrac{1}{\varepsilon}$, then taking integration leads to
	\begin{equation}\label{eq:cn2}
	\frac{1}{\varepsilon}(F(\phi^{n+1})-F(\phi^{n})-f(\hat{\phi}^{n+\frac{1}{2}})\delta_t\phi^{n+1}, 1) 
	\le
	\frac{L}{4\varepsilon}(\|\delta_t \phi^{n+1}\|^2 + \|\delta_{tt}\phi^{n+1}\|^2
	+ \|\delta_t\phi^n\|^2).
	\end{equation}
	For the term involving $B$, by using identity
	\eqref{eq:ID:1} with $h^{n+1}=\delta_t \phi^{n+1}$, one get
	\begin{equation}\label{cs11-0}
	-B(\delta_{tt}\phi^{n+1},\delta_t\phi^{n+1})
	=-\frac{B}{2}\|\delta_{t}\phi^{n+1}\|^{2}+\frac{B}{2}\|\delta_{t}\phi^{n}\|^{2}
	-\frac{B}{2}\|\delta_{tt}\phi^{n+1}\|^{2}.
	\end{equation}
	Summing up \eqref{cn5}-\eqref{cs11-0}, we obtain
	\begin{equation}\label{cn8}
	\begin{split}
	&\frac{\varepsilon}{2}(\|\nabla \phi^{n+1}\|^{2} -
	\|\nabla \phi^{n}\|^{2}) 
	+ \frac{1}{\varepsilon}(F(\phi^{n+1})-F(\phi^{n}),1)
	+
	\frac{B}{2}\|\delta_{t}\phi^{n+1}\|^{2} -\frac{B}{2}\|\delta_{t}\phi^{n}\|^{2}\\
	\leq & -
	2\sqrt{\frac{A}{\gamma}}\|\delta_{t}\phi^{n+1}\|^{2}
	+\frac{L}{4\varepsilon}\|\delta_{t} \phi^{n+1}\|^{2}
	+\frac{L}{4\varepsilon}\|\delta_{t} \phi^{n}\|^{2}
	-\frac{B}{2}\|\delta_{tt}\phi^{n+1}\|^{2}
	+\frac{L}{4\varepsilon}\|\delta_{tt} \phi^{n+1}\|^{2},
	\end{split}
	\end{equation}
	which is the energy estimate \eqref{eq:CN:Edis}.
\end{proof}

\begin{remark}\label{rem:stab1}
  Note that, if $B=0$, we can take
  $A\geq \dfrac{L^{2}\gamma}{4\varepsilon^{2}}$ to make the
  SL-CN scheme (\ref{eq:CN:1})-(\ref{eq:CN:2})
  unconditionally stable as well.  However, when $A=0$, we
  can't prove an unconditional stability for
  $B\sim O(\varepsilon^{-1})$ or
  $B\sim O(\varepsilon^{-2})$.
\end{remark}

\begin{remark}\label{rmk:stab2}
  The constant $A$ defined in equation \eqref{eq:CN:ABcond}
  seems to be quite large when $\varepsilon$ is small, but
  it is not necessarily true. Since usually $\gamma$ is a
  small constant related to $\varepsilon$. For example, it
  was pointed out in \cite{magaletti2013sharp} that, the
  Cahn-Hilliard equation coupled with the Navier-Stokes
  equations have a sharp-interface limit when
  $O(\varepsilon^3) \le \gamma \le O(\varepsilon)$, while
  $\gamma \sim O(\varepsilon^2)$ gives the fastest
  convergence. On the other hand, the numerical results in
  Section 4 shows that in practice $A$ can take much smaller
  values than those defined in \eqref{eq:CN:ABcond} when
  nonzero $B$ values are used.
\end{remark}

\begin{remark}\label{rmk:stab3}
  The discrete Energy $E_C$ defined in equation
  \eqref{eq:CN:E} is a first order approximation to the
  original energy $E$, since
  $ \| \delta_t \phi^{n+1} \|^2 \sim O(\tau^2)$.  On the
  other side, summing up the equation \eqref{eq:CN:Edis} for
  $n=1,\ldots, N$, we get
	\begin{equation}\label{eq:steady}
	E^{N+1}_{CN} + \sum_{n=1}^N \left( 
	\Big(2\sqrt{\frac{A}{\gamma}}-\frac{L}{2\varepsilon}\Big)\|\delta_{t}\phi^{n+1}\|^{2}
	-\Big(\frac{B}{2}-\frac{L}{4\varepsilon}\Big)\|\delta_{tt}\phi^{n+1}\|^{2}
	\right) \leq E^1_{CN}.
	\end{equation}
	By taking $N\rightarrow\infty$, we get
    $\delta_t \phi^{n+1} \rightarrow 0$, which means the
    system will eventually converge to a steady state. By
    equation \eqref{eq:CN:1} and \eqref{eq:CN:2}, this
    steady state is a critical point of the original energy
    functional $E$.
\end{remark}
\section{Convergence analysis}

In this section, we shall establish error estimate of
the SL-CN scheme. We will shown that, if the interface is well
developed in the initial condition, the error bounds depend
on $\frac{1}{\varepsilon}$ only in some lower polynomial
order for small $\varepsilon$.  Let $\phi(t^{n})$ be the
exact solution at time $t=t^{n}$ to equation of
(\ref{eq:CH}) and $\phi^{n}$ be the solution to the time discrete numerical scheme
(\ref{eq:CN:1})-(\ref{eq:CN:2}), we define error function
$e^{n}:=\phi^{n}-\phi(t^{n})$. Obviously $e^{0}=0$.

Before presenting the detailed error analysis, we first make
some assumptions. For simplicity, we take $\gamma=1$ in this
section, and assume $0<\varepsilon<1$.  We use notation
$\lesssim$ in the way that $f\lesssim g$ means that
$f \le C g$ with positive constant $C$ independent of $\tau$
and $\varepsilon$.
\begin{ap}\label{ap:1}
	We make following assumptions on $f$:
	\begin{enumerate}
		\item [(1)]$F\in C^{4}(\mathbf{R})$, 
		$F(\pm 1)=0$, and $F>0$ elsewhere. There exist two
		non-negative constants $B_0,B_1$, such that
		\begin{equation}\label{eq:AP:Fcoercive} \phi^2 \le
		B_0 + B_1 F(\phi),\quad \forall\; \phi\in\mathbf{R}.
		\end{equation}
		\item[(2)] $f=F'$. $f'$ and $f''$ are uniformly bounded, or,  $f$
		satisfies \eqref{eq:Lip} and
		\begin{equation}\label{eq:Lip2}
		\max_{\phi\in\mathbf{R}} | f''(\phi) | \le L_2,
		\end{equation}
		where $L_2$ is a non-negative constant.
	\end{enumerate}
\end{ap}

\begin{ap}\label{ap:02}
We assume that there exist positive constants $m_{0}$ and non-negative constants $\sigma_{1},\sigma_{2},\sigma_{3}$ such that
\begin{equation}\label{ap:021}
  m_{0}:=\frac{1}{|\Omega|}\int_{\Omega}\phi^{0}(x){\rm d}x \in (-1,1),
\end{equation}
\begin{equation}\label{ap:022}
 E(\phi^{0}):=\frac{\varepsilon}{2}\|\nabla \phi^{0}\|^{2}+\frac{1}{\varepsilon}\|F(\phi^{0})\|_{L^{1}}\lesssim \varepsilon^{-2\sigma_{1}}.
\end{equation}
\begin{equation}\label{ap:022-1}
 \|\mu^{0}\|_{H^{l}}:=\|-\varepsilon \Delta \phi^{0}+ \frac{1}{\varepsilon}f(\phi^{0})\|_{H^{l}}
 \lesssim \varepsilon^{-2\sigma_{2+l}}, \ l=0,1.
\end{equation}
We also assume that an appropriate scheme is used to
calculate the numerical solution at first step, such
that
\begin{equation}\label{eq:AP:phi1}
  m_1:=\frac{1}{|\Omega|}\int_{\Omega}\phi^{1}(x){\rm d}x = m_0, \qquad
  E_\varepsilon (\phi^1) \le E_\varepsilon (\phi^0) \lesssim\varepsilon^{-2\sigma_{1}},
  \end{equation}
\begin{equation} \label{eq:AP:phi1mu2}
\|\delta_{t}\phi^{1}\|^{2} \lesssim \varepsilon^{-2 \sigma_{1}},
\end{equation}
 then
\begin{equation}\label{ap:0220}
  E_{CN}^{1}\lesssim \varepsilon^{-2 \sigma_{1}}+  \varepsilon^{-2\sigma_{1}-1},
\end{equation}
and exist a constant $\sigma_{0}>0$,
\begin{equation}\label{ap:022-2}
  \|e^{1}\|_{-1}^{2}+\|e^{1}\|^{2}+\varepsilon \|\nabla e^{1}\|^{2} \lesssim \varepsilon^{-\sigma_{0}}\tau^{4}.
\end{equation}
\end{ap}

\begin{lm}\label{lm:02}
  Suppose that $f$ satisfies Assumption \ref{ap:1},
  $\phi_{0}\in H^{2}(\Omega)$. Then,
  the following estimates holds for the numerical solution
  of \eqref{eq:CN:1}-\eqref{eq:CN:2}
  \begin{equation}\label{eq:CA:conserv}
	\frac{1}{|\Omega|}\int_{\Omega}\phi^{n}(x){\rm d}x=m_0,\quad n=1,\ldots, N+1,
  \end{equation}
\begin{equation}\label{eq:lm:02}
  E_{CN}^{n+1} \leq E_{CN}^{1} \lesssim \varepsilon^{-2 \sigma_{1}-1}.
\end{equation}
\end{lm}
\begin{proof} (i) Equation \eqref{eq:CA:conserv} is obtained
  by integrating equation \eqref{eq:CN:1}.
	
  (ii) Equation \eqref{eq:lm:02}
  is a direct result of the energy estimate
  \eqref{eq:CN:Edis} and \eqref{ap:0220}.
\end{proof}

\indent Some regularities of exact solution $\phi(t)$ are necessary for the error estimates.
\begin{ap}\label{ap:03}
Suppose the exact solution of (\ref{eq:CH}) have the following regularities:\\
\begin{enumerate}
  \item[(1)] $\Delta^{-1} \phi(t) \in W^{2,2}(0,\infty;H^{-1})$, or
  $$\int_{0}^{\infty}\|\partial_{tt}\Delta^{-1}\phi(t)\|_{-1}^{2}{\rm d}t \leq \varepsilon^{-\rho_{1}},$$
  \item[(2)] $\phi(t) \in W^{2,2}(0,\infty;H^{-1}\bigcap H^{3})$, or
  \begin{align*}
 \int_{0}^{\infty}\|\partial_{tt}\phi(t)\|_{-1}^{2}{\rm d}t &\leq \varepsilon^{-\rho_{2}},\quad&
 \int_{0}^{\infty}\|\partial_{tt} \phi(t)\|^{2}{\rm d}t &\leq \varepsilon^{-\rho_{3}},& \\
 \int_{0}^{\infty}\|\partial_{tt} \nabla \phi(t)\|^{2}{\rm d}t &\leq \varepsilon^{-\rho_{4}},\quad&
 \int_{0}^{\infty}\|\partial_{tt}\nabla \Delta \phi(t)\|^{2}{\rm d}t &\leq \varepsilon^{-\rho_{5}},&
  \end{align*}
 \item[(3)] $\phi(t) \in W^{1,2}(0,\infty;H^{3})$, or
   $$\int_{0}^{\infty}\|\partial_{t}\nabla \phi(t)\|^{2}{\rm d}t \leq \varepsilon^{-\rho_{6}},\quad
  \int_{0}^{\infty}\|\partial_{t}\nabla \Delta \phi(t)\|^{2}{\rm d}t \leq \varepsilon^{-\rho_{7}},\quad
  \int_{0}^{\infty}\|\partial_{t}\phi(t)\|^{2}{\rm d}t \leq \varepsilon^{-\rho_{8}}.$$
\end{enumerate}
Here $\rho_{j}, j=1,2,3,4,5,6,7,8$ are non-negative constants which depend on $\sigma_{1}, \sigma_{2}, \sigma_{3}.$
\end{ap}

We first carry out a coarse error estimate using a
standard approach for time semi-discretized schemes.
\begin{prop}(Coarse error estimate)\label{prio}
Suppose that $A, B$ are any non-negative number. Then for all $N\geq1$, we have estimate
\begin{equation}\label{cet2-3}
  \begin{split}
   &\|e^{N+1}\|_{-1}^{2}
   +\frac{\varepsilon \tau}{4}\|\nabla \frac{e^{N+1}+e^{N}}{2}\|^{2}
   +A\tau^{2}\|\nabla e^{N+1}\|^{2}
   +B\tau\|e^{N+1}\|^{2}\\
   \lesssim &
   \varepsilon^{-\max\{\rho_{1}+1, \rho_{2}+3, \rho_{4}-1, \rho_{6}+5\}} \tau^{4} + A\tau^{2}\|\nabla e^{N}\|^{2}
  +B\tau\|e^{N}\|^{2}\\
  &+\Big{(}1+\frac{4B^{2}\tau}{\varepsilon}+\frac{9L^{2}\tau}{\varepsilon^{3}}\Big{)}\|e^{N}\|_{-1}^{2}
  +\Big{(}\frac{4B^{2}\tau}{\varepsilon}+\frac{L^{2}\tau}{\varepsilon^{3}}\Big{)}\|e^{N-1}\|^{2}_{-1},\quad \forall\ \tau>0.\\
   \end{split}
 \end{equation}
 and
 \begin{equation}\label{cet5-1}
  \begin{split}
   &\max_{1\leq n\leq N}\Big{(}\|e^{n+1}\|_{-1}^{2}
   +A\tau^{2}\|\nabla e^{n+1}\|^{2}
   +B\tau\|e^{n+1}\|^{2}\Big{)}
   +\frac{\varepsilon \tau}{4}\sum_{n=1}^{N}\|\nabla \frac{e^{n+1}+e^{n}}{2}\|^{2}\\
  \lesssim &
  \mathrm{exp } \Big{(}\frac{16B^{2}T}{\varepsilon}+\frac{20L^{2}T}{\varepsilon^{3}}\Big{)}
   \varepsilon^{-\max\{\rho_{1}+1, \rho_{2}+3, \rho_{4}-1, \rho_{6}+5, \sigma_{0}+3\}}\tau^{4},\quad\forall\ \tau<1.
   \end{split}
 \end{equation}
 The index $\sigma_0+3$ in \eqref{cet5-1} can be replaced with $\sigma_0$ if we take $\tau<\varepsilon^{1.5}$.
\end{prop}

  \begin{proof}
  The following equations for the error function hold:
  \begin{equation}\label{cet}
 \frac{e^{n+1}- e^{n}}{\tau}
 =\Delta(\mu^{n+\frac{1}{2}}-\mu(t^{n+\frac{1}{2}}))
 +\Big{(}\phi^{n+\frac{1}{2}}_{t}-\frac{\phi(t^{n+1})-\phi(t^{n})}{\tau}\Big{)},
 \end{equation}
\begin{equation}\label{cet1}
\begin{split}
 \mu^{n+\frac{1}{2}}-\mu(t^{n+\frac{1}{2}})=
&-\varepsilon\Delta \Big{(} \frac{\phi^{n+1}+\phi^{n}}{2}-\phi(t^{n+\frac{1}{2}}) \Big{)} 
+\frac{1}{\varepsilon}\Big{(}f(\frac{3}{2}\phi^{n}-\frac{1}{2}\phi^{n-1})-f(\phi(t^{n+\frac{1}{2}}))\Big{)}\\
 &-A\tau\Delta \delta_{t}\phi^{n+1}
 +B\delta_{tt}\phi^{n+1}.
 \end{split}
\end{equation}
Pairing (\ref{cet}) with $-\Delta^{-1}\big{(} \frac{e^{n+1}+e^{n}}{2}\big{)}$, adding (\ref{cet1}) paired with $-\big{(} \frac{e^{n+1}+e^{n}}{2}\big{)}$, we get
\begin{equation}\label{cet2}
  \begin{split}
   &\frac{1}{2\tau }(\|e^{n+1}\|_{-1}^{2}-\|e^{n}\|_{-1}^{2})
   +\varepsilon\|\nabla \frac{e^{n+1}+e^{n}}{2}\|^{2}
   +\frac{A\tau}{2}(\|\nabla e^{n+1}\|^{2}-\|\nabla e^{n}\|^{2})\\
   =&-\Big{(}R_{1}^{n+1},\Delta ^{-1 }\frac{e^{n+1}+e^{n}}{2}\Big{)}
   +A\Big{(}\Delta R_{2}^{n+1},\frac{e^{n+1}+e^{n}}{2}\Big{)}
   -B \Big{(}R_{3}^{n+1},\frac{e^{n+1}+e^{n}}{2}\Big{)}\\
   &+\varepsilon \Big{(} \Delta R_{4}^{n+1},\frac{e^{n+1}+e^{n}}{2}\Big{)}
   -B\Big{(}\delta_{tt}e^{n+1},\frac{e^{n+1}+e^{n}}{2}\Big{)}\\
   &-\frac{1}{\varepsilon}\Big{(}f(\frac{3}{2}\phi^{n}-\frac{1}{2}\phi^{n-1})-f(\phi(t^{n+\frac{1}{2}})),\frac{e^{n+1}+e^{n}}{2}\Big{)}\\
  =&:J_{1}+J_{2}+J_{3}+J_{4}+J_{5}+J_{6}=:J,\\
   \end{split}
 \end{equation}
where
\begin{equation}\label{e3}
  R_{1}^{n+1}=\phi_{t}^{n+\frac{1}{2}}-\frac{\phi(t^{n+1})-\phi(t^{n}))}{\tau},
\end{equation}
\begin{equation}\label{e5}
  R_{2}^{n+1}=\tau(\phi(t^{n+1})-\phi(t^{n})),
\end{equation}
\begin{equation} \label{e4}
{R_{3}^{n+1}=}
 \phi(t^{n+1})-2\phi(t^{n})+\phi(t^{n-1}),
\end{equation}
\begin{equation} \label{e4-4}
{R_{4}^{n+1}=}
 \frac{\phi(t^{n+1})+\phi(t^{n})}{2}-\phi(t^{n+\frac{1}{2}}).
\end{equation}
For the right hand of (\ref{cet2}), by using Cauchy-Schwartz inequality, we obtain the following estimate:
\begin{equation}\label{ej1}
J_{1}=-\Big{(}R_{1}^{n+1},\Delta ^{-1 }\frac{e^{n+1}+e^{n}}{2}\Big{)}
\leq \frac{1}{\eta}\|\Delta^{-1}R_{1}^{n+1}\|_{-1}^{2}+\frac{\eta}{4}\|\nabla \frac{e^{n+1}+e^{n}}{2}\|^{2},
\end{equation}

\begin{equation}\label{ej2}
J_{2}=A\Big{(}\Delta R_{2}^{n+1},\frac{e^{n+1}+e^{n}}{2}\Big{)}
\leq \frac{A^{2}}{\eta}\|\nabla R_{2}^{n+1}\|^{2}+\frac{\eta}{4}\|\nabla \frac{e^{n+1}+e^{n}}{2}\|^{2},
\end{equation}

\begin{equation}\label{ej3}
J_{3}=-B \Big{(}R_{3}^{n+1},\frac{e^{n+1}+e^{n}}{2}\Big{)}
\leq \frac{B^{2}}{\eta}\|R_{3}^{n+1}\|_{-1}^{2}+\frac{\eta}{4}\|\nabla \frac{e^{n+1}+e^{n}}{2}\|^{2},
\end{equation}

\begin{equation}\label{ej4}
J_{4}=\varepsilon \Big{(} \Delta R_{4}^{n+1},\frac{e^{n+1}+e^{n}}{2}\Big{)}
\leq \frac{\varepsilon^{2}}{\eta}\|\nabla R_{4}^{n+1}\|^{2}+\frac{\eta}{4}\|\nabla \frac{e^{n+1}+e^{n}}{2}\|^{2}.
\end{equation}
For $J_{5}$ of the right side of (\ref{cet2}), by using 
$\delta_{tt}e^{n+1}=\delta_{t}e^{n+1}-\delta_{t}e^{n}$, we have
\begin{equation}\label{e11}
\begin{split}
J_{5}&=-B \Big{(}\delta_{tt}e^{n+1},\frac{e^{n+1}+e^{n}}{2}\Big{)}\\
&=-\frac{B}{2}(\|e^{n+1}\|^{2}-\|e^{n}\|^{2})+\frac{B}{2}(\delta_{t}e^{n},e^{n+1}+e^{n})\\
&\leq-\frac{B}{2}(\|e^{n+1}\|^{2}-\|e^{n}\|^{2})+\frac{B^{2}}{\eta}\|e^{n}-e^{n-1}\|^{2}_{-1}+ \frac{\eta}{4}\|\nabla \frac{e^{n+1}+e^{n}}{2}\|^{2}.
\end{split}
\end{equation}
\begin{equation}\label{e12}
\begin{split}
J_{6}&=-\frac{1}{\varepsilon}\Big{(}f(\frac{3}{2}\phi^{n}-\frac{1}{2}\phi^{n-1})-f(\phi(t^{n+\frac{1}{2}})),\frac{e^{n+1}+e^{n}}{2}\Big{)}\\
&\leq \frac{L}{\varepsilon}\Big{(}| R_{5}^{n+1}|,|\frac{e^{n+1}+e^{n}}{2}|\Big{)}+ \frac{L}{\varepsilon}\Big{(} |\frac{3}{2}e^{n}-\frac{1}{2}e^{n-1}|,|\frac{e^{n+1}+e^{n}}{2}|\Big{)}\\
&\leq \frac{L^{2}}{\varepsilon^{2}\eta}\Big{(}\|R_{5}^{n+1}\|_{-1}^{2} +\| \frac{3}{2}e^{n}-\frac{1}{2}e^{n-1} \|_{-1}^{2} \Big{)}+\frac{\eta}{2}\|\nabla \frac{e^{n+1}+e^{n}}{2}\|^{2},
\end{split}
\end{equation}
where
\begin{equation} \label{e5-5}
{R_{5}^{n+1}=}
 \frac{3}{2}\phi(t^{n})-\frac{1}{2}\phi(t^{n-1})-\phi(t^{n+\frac{1}{2}}).
\end{equation}
For the $R_1, \ldots, R_5$ terms, we have following estimates:
\begin{equation} \label{e6}
\|\Delta^{-1}R_{1}^{n+1}\|_{-1}^{2}  \lesssim
 \tau^{3}\int_{t^{n}}^{t^{n+1}}\|\partial_{tt}\Delta^{-1}\phi(t)\|_{-1}^{2}{\rm d}t,
\end{equation}
\begin{equation} \label{e8}
\|\nabla R_{2}^{n+1}\|^{2}  \lesssim
 \tau^{3}\int_{t^{n}}^{t^{n+1}}\|\partial_{t}\nabla \phi(t)\|^{2}{\rm d}t,
\end{equation}
\begin{equation}\label{e7}
  \|R_{3}^{n+1}\|_{-1}^{2}  \lesssim 6\tau^{3}\int_{t^{n-1}}^{t^{n+1}}\|\partial_{tt}\phi(t)\|_{-1}^{2}{\rm d}t,
\end{equation}
\begin{equation}\label{e7-1}
  \|\nabla R_{4}^{n+1}\|^{2}   \lesssim
   \tau^{3}\int_{t^{n}}^{t^{n+1}}\|\partial_{tt} \nabla \phi(t)\|^{2}{\rm d}t,
\end{equation}
\begin{equation}\label{e7-2}
  \|R_{5}^{n+1}\|_{-1}^{2}  \lesssim \tau^{3}\int_{t^{n-1}}^{t^{n+1}}\|\partial_{tt}\phi(t)\|_{-1}^{2}{\rm d}t.
\end{equation}
Substituting $J_{1},\cdots, J_{6}$ into (\ref{cet2}), we have
\begin{equation}\label{cet2-1}
  \begin{split}
   &\frac{1}{2\tau }(\|e^{n+1}\|_{-1}^{2}-\|e^{n}\|_{-1}^{2})
   +\varepsilon\|\nabla \frac{e^{n+1}+e^{n}}{2}\|^{2}
   +\frac{A\tau}{2}(\|\nabla e^{n+1}\|^{2}-\|\nabla e^{n}\|^{2})\\
   &+\frac{B}{2}(\|e^{n+1}\|^{2}-\|e^{n}\|^{2})\\
  \lesssim & \frac{1}{\eta}C_{1}^{n+1}\tau^{3}+\frac{7\eta}{4}\|\nabla \frac{e^{n+1}+e^{n}}{2}\|^{2}
  +\frac{B^{2}}{\eta}\|e^{n}-e^{n-1}\|^{2}_{-1}
  +\frac{L^{2}}{\varepsilon^{2}\eta}\| \frac{3}{2}e^{n}-\frac{1}{2}e^{n-1} \|_{-1}^{2},
   \end{split}
 \end{equation}
 where
 \begin{equation}\notag
 \begin{split}
   C_{1}^{n+1}=&\int_{t_{n}}^{t_{n+1}}(\|\partial_{tt}\Delta^{-1}\phi(t)\|_{-1}^{2}
+A^{2}\|\partial_{t}\nabla \phi(t)\|^{2}){\rm d}t\\
&+\int_{t_{n-1}}^{t_{n+1}}\Big{(}\Big{(}B^{2}+\frac{L^{2}}{\varepsilon^{2}}\Big{)}\|\partial_{tt}\phi(t)\|_{-1}^{2}
+\varepsilon^{2}\|\partial_{tt}\nabla \phi(t)\|^{2}\Big{)}{\rm d}t.
\end{split}
 \end{equation}
 Taking $\eta=\varepsilon/2 $, multiplying \eqref{cet2-1} by
 $2\tau$, we obtain \eqref{cet2-3} by using inequality
 $\|a+b\|^2\leq 2\|a\|^2+2\|b\|^2$ and estimates
 \eqref{e6}-\eqref{e7-2}.  Then by summing (\ref{cet2-1})
 for $n=1\cdots N$, we obtain
\begin{equation}\label{cet5}
  \begin{split}
   &(\|e^{N+1}\|_{-1}^{2}-\|e^{1}\|_{-1}^{2})
   +\frac{\varepsilon \tau}{4}\sum_{n=1}^{N}\|\nabla \frac{e^{n+1}+e^{n}}{2}\|^{2}
   +A\tau^{2}(\|\nabla e^{N+1}\|^{2}-\|\nabla e^{1}\|^{2})\\
   &+B\tau(\|e^{N+1}\|^{2}-\|e^{1}\|^{2})\\
  \lesssim & \frac{4}{\varepsilon}C_{1}\tau^{4}
  +\Big{(}\frac{16B^{2}\tau}{\varepsilon}+\frac{20L^{2}\tau}{\varepsilon^{3}}\Big{)}\sum_{n=1}^{N}\|e^{n}\|_{-1}^{2},\\
   \end{split}
 \end{equation}
where
\begin{equation}\label{cet5-0}
\begin{split}
C_{1}=&\sum_{n=1}^{N}C_{1}^{n+1}\\
\leq&\int_{0}^{T}\Big{(}\|\partial_{tt}\Delta^{-1}\phi(t)\|_{-1}^{2}
+A^{2}\|\partial_{t}\nabla \phi(t)\|^{2}
+2\Big{(}B^{2}+\frac{L^{2}}{\varepsilon^{2}}\Big{)}\|\partial_{tt}\phi(t)\|_{-1}^{2}\\
&+2\varepsilon^{2}\|\partial_{tt}\nabla\phi(t)\|^{2}\Big{)}{\rm d}t\\
\lesssim & \varepsilon^{-\max\{\rho_{1}, \rho_{2}+2, \rho_{4}-2, \rho_{6}+4\}}.
\end{split}
\end{equation}
by discrete Gronwall inequality and assumption \eqref{ap:022-2}, we get (\ref{cet5-1}).
\end{proof}

Proposition \ref{prio} is the usual error estimate, in
which the error growth depends on $T/\varepsilon^3$
exponentially. To obtain a finer estimate on the error, we
need to use a spectral estimate of the linearized
Cahn-Hilliard operator by Chen \cite{chen_spectrum_1994} for
the case when the interface is well developed in the
initial condition.
\begin{lm}\label{prop:01}
  Let $\phi(t)$ be the exact solution of the Cahn-Hilliard
  equation \eqref{eq:CH} with interfaces are well developed
  in the initial condition (i.e. conditions (1.9)-(1.15) in
  \cite{chen_spectrum_1994} are satisfied).  Then there
  exist $0<\varepsilon_{0}\ll 1$ and positive constant
  $C_{0}$ such that the principle eigenvalue of the
  linearized Cahn-Hilliard operator
  $\mathcal{L}_{CH}:=\Delta(\varepsilon\Delta-\frac{1}{\varepsilon}f'(\phi)I)$
  satisfies for all $t\in [0,T]$
  \begin{equation}\label{eigen}
    \lambda_{CH}=\inf_{\substack{0\neq v\in H^{1}(\Omega)\\ \Delta\omega=v}}
    \frac{\varepsilon\|\nabla v\|^{2}+\frac{1}{\varepsilon}(f'(\phi(\cdot,t))v,v)}{\|\nabla\omega\|^{2}}
    \geq-C_{0},
  \end{equation}
  for $\varepsilon\in (0,\varepsilon_{0})$. 
\end{lm}

The following lemma shows the boundedness of the solution to
the Cahn-Hilliard equation, provided that its sharp-interface limit Hele-Shaw problem has a global (in time)
classical solution. This is a condition of the finer error
estimate.
\begin{lm}\label{lm:Linf}
  Suppose that f satisfies Assumption \ref{ap:1}, and the
  corresponding Hele-Shaw problem has a global (in time)
  classical solution. Then there exists a family of smooth
  initial functions
  $\{\phi_{0}^{\varepsilon}\}_{0< \varepsilon \leq 1}$ and
  constants $ \varepsilon_{0} \in (0,1]$ and $C > 0$ such
  that for all $\varepsilon \in (0,\varepsilon_{0})$ the
  solution $\phi(t)$ of the Cahn-Hilliard equation
  (\ref{eq:CH}) with the above initial data
  $\phi_{0}^{\varepsilon}$ satisfies
\begin{equation}\label{eq:phi}
  \|\phi(t)\|_{L^{\infty}}(0,T; \Omega) \leq C.
\end{equation}
\end{lm}
\begin{proof}
See \cite{feng_numerical_2005} and \cite{alikakos_convergence_1994}
for the detailed proof.
\end{proof}

Now we present the refined error estimate.
\begin{theorem}\label{error}
Suppose all of the Assumption \ref{ap:1},\ref{ap:02},\ref{ap:03} hold and $B>L/2\varepsilon$. Let time
  step $\tau$ satisfy the following constraint
 \begin{equation}\label{error-1}
 \begin{split}
  \tau \lesssim \min \left\{
  \varepsilon^{6},
 \varepsilon^{\frac{1}{18-d}(4\sigma+d +38)}
  \right\},
 \end{split}
 \end{equation}
 then the solution of (\ref{eq:CN:1})-(\ref{eq:CN:2})
  satisfies the following error estimate
\begin{equation}\label{error_2}
   \begin{split}
   &\max_{1\leq n\leq N}\Big(
   \|e^{n+1}\|_{-1}^{2}
   +\tau(A\tau+\varepsilon)\|\nabla e^{n+1}\|^{2}
   +B\tau\|\delta_{t}e^{n+1}\|^{2}
   \Big)
   \\
   &+\sum_{n=1}^{N}\left[\frac{\tau\varepsilon^4}{2}\|\nabla\frac{e^{n+1}+e^{n}}{2}\|^{2}
   +\|\delta_{t}e^{n+1}\|_{-1}^{2}
   +2A\tau^2\|\nabla \delta_{t}e^{n+1}\|^{2}
   +\tau\big(B-\frac{L}{2\varepsilon}\big)\|\delta_{tt}e^{n+1}\|^{2}
   \right]
   \\
   \lesssim{}&\varepsilon^{-\sigma}
   \mathrm{exp }(4(C_{0}+L^{2}+1)T)\tau^{4}.\\
   \end{split}
\end{equation}   
where $\sigma=
\max\{\rho_{1}+4, \rho_{2}+6, \rho_{4}+2, \rho_{5}-8, \rho_{6}+8, \rho_{7}-2, \sigma_0\}$.
\end{theorem}

\begin{proof}
(i) To get a better convergence result, we re-estimate $J_5,J_6$ in \eqref{cet2} as
\begin{equation}\label{e11-2}
 \begin{split}
 J_{5}&= -B \Big{(}\delta_{tt}e^{n+1},\frac{e^{n+1}+e^{n}}{2}\Big{)}
 \leq \frac{B^{2}}{\eta}\|\delta_{tt}e^{n+1}\|^{2}_{-1}+\frac{\eta}{4}\|\nabla \frac{e^{n+1}+e^{n}}{2}\|^{2},
  \end{split}
\end{equation}
\begin{equation}\label{j1}
\begin{split}
  J_{6}=&-\frac{1}{\varepsilon}\Big{(}f(\frac{3}{2}\phi^{n}-\frac{1}{2}\phi^{n-1})-f(\phi(t^{n+\frac{1}{2}})),\frac{e^{n+1}+e^{n}}{2}\Big{)}\\
  =&-\frac{1}{\varepsilon}\Big{(}f(\frac{3}{2}\phi^{n}-\frac{1}{2}\phi^{n-1})-f(\frac{\phi^{n+1}+\phi^{n}}{2}),\frac{e^{n+1}+e^{n}}{2}\Big{)}\\
  &-\frac{1}{\varepsilon}\Big{(}f(\frac{\phi^{n+1}+\phi^{n}}{2})-f(\phi(t^{n+\frac{1}{2}})),\frac{e^{n+1}+e^{n}}{2}\Big{)}\\
  :=&J_{7}+J_{8},
  \end{split}
\end{equation}
\begin{equation}\label{j2}
\begin{split}
  J_{7}
  =&-\frac{1}{\varepsilon}\Big{(}f(\frac{3}{2}\phi^{n}-\frac{1}{2}\phi^{n-1})-f(\frac{\phi^{n+1}+\phi^{n}}{2}),\frac{e^{n+1}+e^{n}}{2}\Big{)}\\
  \leq & \frac{L}{2 \varepsilon}\Big{(}| \delta_{tt}\phi^{n+1}|,|\frac{e^{n+1}+e^{n}}{2}|\Big{)}\\
  =&  \frac{L}{2 \varepsilon}\Big{(}| \delta_{tt}e^{n+1}+R_{3}^{n+1}|,|\frac{e^{n+1}+e^{n}}{2}|\Big{)}\\
  \leq & \frac{L^{2}}{4 \varepsilon^{2}\eta }\|\delta_{tt}e^{n+1}\|_{-1}^{2}+\frac{L^{2}}{4 \varepsilon^{2} \eta}\|R_{3}^{n+1}\|_{-1}^{2} +\frac{\eta}{2}\|\nabla\frac{e^{n+1}+e^{n}}{2}\|^{2}.
  \end{split}
\end{equation}
For $J_8$, by Taylor expansion, there exists $\vartheta^{n+1}$ between  $\frac{\phi^{n+1}+\phi^{n}}{2}$ and $\phi(t^{n+\frac{1}{2}})$, such that
\begin{equation}\label{j3}
\begin{split}
  J_{8}
  =&-\frac{1}{\varepsilon}\Big{(}f(\frac{\phi^{n+1}+\phi^{n}}{2})-f(\phi(t^{n+\frac{1}{2}})),\frac{e^{n+1}+e^{n}}{2}\Big{)}\\
  =&-\frac{1}{\varepsilon}\Big{(}f'(\phi(t^{n+\frac{1}{2}})) \Big{(}\frac{e^{n+1}+e^{n}}{2}+R_{4}^{n+1}\Big{)},\frac{e^{n+1}+e^{n}}{2}\Big{)}\\
  &-\frac{1}{2\varepsilon}\Big{(}f''(\vartheta^{n+1}) \Big{(}\frac{e^{n+1}+e^{n}}{2}+R_{4}^{n+1}\Big{)}^{2},\frac{e^{n+1}+e^{n}}{2}\Big{)}\\
  \leq &-\frac{1}{\varepsilon}\Big{(}f'(\phi(t^{n+\frac{1}{2}})) \frac{e^{n+1}+e^{n}}{2},\frac{e^{n+1}+e^{n}}{2}\Big{)}
  +\frac{L_{2}}{\varepsilon}\|\frac{e^{n+1}+e^{n}}{2}\|^{3}_{L^{3}}\\
  &+\frac{1}{ \varepsilon^{2}\eta}C_{2}\|R_{4}^{n+1}\|_{-1}^{2}+\frac{\eta}{2}\|\nabla \frac{e^{n+1}+e^{n}}{2}\|^{2},
  \end{split}
 \end{equation}
 where $ C_{2}=L^{2}+4L^{2}_{2}\| \phi(t)\|_{\infty}^{2} \leq L^{2}+4L^{2}_{2} C^{2}$. 
Here we assume that the conditions of Lemma \ref{lm:Linf} are satisfied. 

\indent Substituting $J_{1},\cdots, J_{8}$ into (\ref{cet2}), then we have
\begin{equation}\label{j4}
  \begin{split}
   &\frac{1}{2\tau }(\|e^{n+1}\|_{-1}^{2}-\|e^{n}\|_{-1}^{2})
   +\varepsilon\|\nabla \frac{e^{n+1}+e^{n}}{2}\|^{2}
   +\frac{A\tau}{2}(\|\nabla e^{n+1}\|^{2}-\|\nabla e^{n}\|^{2})\\
  \leq & \frac{1}{\eta}\|\Delta^{-1}R_{1}^{n+1}\|_{-1}^{2}+\frac{A^{2}}{\eta}\|\nabla R_{2}^{n+1}\|^{2}
  +\Big{(}\frac{B^{2}}{\eta} +\frac{L^{2}}{4 \varepsilon^{2} \eta}\Big{)}\|R_{3}^{n+1}\|_{-1}^{2}
   +\frac{\varepsilon^{2}}{\eta}\|\nabla R_{4}^{n+1}\|^{2}\\
     &+\frac{1}{ \varepsilon^{2}\eta}C_{2}\|R_{4}^{n+1}\|_{-1}^{2}
     +\frac{9}{4}\eta\|\nabla \frac{e^{n+1}+e^{n}}{2}\|^{2}
   +\Big{(}\frac{B^{2}}{\eta}+ \frac{L^{2}}{4 \varepsilon^{2}\eta }\Big{)}\|\delta_{tt}e^{n+1}\|^{2}_{-1}\\
   &-\frac{1}{\varepsilon}\Big{(}f'(\phi(t^{n+\frac{1}{2}})) \frac{e^{n+1}+e^{n}}{2},\frac{e^{n+1}+e^{n}}{2}\Big{)}
  +\frac{L_{2}}{\varepsilon}\|\frac{e^{n+1}+e^{n}}{2}\|^{3}_{L^{3}}.\\
   \end{split}
 \end{equation}
We need to bound the last three terms on the right hand side of above inequality.

(ii) To control the $\|\delta_{tt}e^{n+1} \|_{-1}^2$ term, 
we pair (\ref{cet}) with $-\Delta^{-1}\delta_{t}e^{n+1}$, then add (\ref{cet1}) paired with $-\delta_{t}e^{n+1}$, to get
\begin{equation}\label{t0}
  \begin{split}
   &\frac{1}{\tau}\|\delta_{t}e^{n+1}\|_{-1}^{2}
   +\frac{\varepsilon}{2}(\|\nabla e^{n+1}\|^{2}-\|\nabla e^{n}\|^{2})
   +A\tau\|\nabla \delta_{t}e^{n+1}\|^{2}\\
   &+\frac{B}{2}(\|\delta_{t}e^{n+1}\|^{2}-\|\delta_{t}e^{n}\|^{2}+\|\delta_{tt}e^{n+1}\|^{2})\\
   =&-(R_{1}^{n+1},\Delta ^{-1 }\delta_{t}e^{n+1})
   +A(\Delta R_{2}^{n+1},\delta_{t}e^{n+1})
   -B (R_{3}^{n+1},\delta_{t}e^{n+1})\\
   &+\varepsilon ( \Delta R_{4}^{n+1},\delta_{t}e^{n+1})
   -\frac{1}{\varepsilon}\Big{(}f(\frac{3}{2}\phi^{n}-\frac{1}{2}\phi^{n-1})-f(\phi(t^{n+\frac{1}{2}})),\delta_{t}e^{n+1}\Big{)}\\
  =&:\widetilde{J}_{1}+\widetilde{J}_{2}+\widetilde{J}_{3}+\widetilde{J}_{4}+\widetilde{J}_{5}=:\widetilde{J},\ \ \ \rm{n\geq1}.\\
   \end{split}
 \end{equation}
  Analogously, applying the method for $J_{1}, \cdots, J_{4}$ to $\widetilde{J}_{1}, \cdots, \widetilde{J}_{4}$, yields
 \begin{equation}\label{t1}
  \begin{split}
  \widetilde{J}_{1}=&-(R_{1}^{n+1},\Delta ^{-1 }\delta_{t}e^{n+1})
  \leq \frac{1}{\tilde{\eta}}\|R_{1}^{n+1}\|_{-1}^{2}+\frac{\tilde{\eta}}{4}\|\delta_{t}e^{n+1}\|^{2}_{-1},
   \end{split}
 \end{equation}
  \begin{equation}\label{t2}
  \begin{split}
  \widetilde{J}_{2}=&A(\Delta R_{2}^{n+1},\delta_{t}e^{n+1})
  \leq \frac{A^{2}}{\tilde{\eta}}\|\nabla \Delta R_{2}^{n+1}\|^{2}+\frac{\tilde{\eta}}{4}\|\delta_{t}e^{n+1}\|^{2}_{-1},
   \end{split}
 \end{equation}
 \begin{equation}\label{t3}
  \begin{split}
  \widetilde{J}_{3}=&-B (R_{3}^{n+1},\delta_{t}e^{n+1})
  \leq \frac{B^{2}}{\tilde{\eta}}\|\nabla R_{3}^{n+1}\|^{2}+\frac{\tilde{\eta}}{4}\|\delta_{t}e^{n+1}\|^{2}_{-1},
   \end{split}
 \end{equation}
 \begin{equation}\label{t4}
  \begin{split}
  \widetilde{J}_{4}=&\varepsilon ( \Delta R_{4}^{n+1},\delta_{t}e^{n+1})
  \leq \frac{\varepsilon^{2}}{\tilde{\eta}}\|\nabla \Delta R_{4}^{n+1}\|^{2}+\frac{\tilde{\eta}}{4}\|\delta_{t}e^{n+1}\|^{2}_{-1}.
   \end{split}
 \end{equation}
 For $\widetilde{J}_{5}$ of (\ref{t0}), we have
  \begin{equation}\label{t5}
     \begin{split}
    \widetilde{J}_{5}=
   &-\frac{1}{\varepsilon}\Big{(}f(\frac{3}{2}\phi^{n}-\frac{1}{2}\phi^{n-1})-f(\phi(t^{n+\frac{1}{2}})),\delta_{t}e^{n+1}\Big{)}\\
  \leq &
  -\frac{1}{\varepsilon}\Big{(}f'(\xi^{n+1})\Big{(}
  -\frac{1}{2}\delta_{tt} e^{n+1}- \frac{1}{2}R_{3}^{n+1}
  +\frac{e^{n+1}+e^{n}}{2}+R_{4}^{n+1}\Big{)},\delta_{t}e^{n+1} \Big{)}\\
  \leq & \frac{1}{2 \varepsilon}(f'(\xi^{n+1})\delta_{tt}e^{n+1}, \delta_{t}e^{n+1})+
  \frac{L^{2}}{4\varepsilon^{2}\tilde{\eta}}\|\nabla R_{3}^{n+1}\|^{2}+\frac{\tilde{\eta}}{4}\|\delta_{t}e^{n+1}\|^{2}_{-1}\\
  &+\frac{\eta}{4}\|\nabla \frac{e^{n+1}+e^{n}}{2}\|^{2}+\frac{L^{2}}{ \varepsilon^{2}\eta }\|\delta_{t}e^{n+1}\|^{2}_{-1}
  +\frac{L^{2}}{\varepsilon^{2}\tilde{\eta}}\|\nabla R_{4}^{n+1}\|^{2}+\frac{\tilde{\eta}}{4}\|\delta_{t}e^{n+1}\|^{2}_{-1},\\
   \end{split}
 \end{equation}
 where $\xi^{n+1}$ is a fixed number between $\frac{3}{2} \phi^{n} -\frac{ 1}{2}\phi^{n-1}$ and $\phi(t^{n+\frac{1}{2}})$.
Now, we estimate the first term on the right hand side of (\ref{t5}).
\begin{equation}\label{1}
\begin{split}
  &\dfrac{1}{2\varepsilon} (f'(\xi^{n+1}) \delta_{tt} e^{n+1}, \delta_{t}e^{n+1} )\\
  =& \dfrac{1}{4 \varepsilon} (f'(\xi^{n+1}),( \delta_{t} e^{n+1})^{2}-( \delta_{t}e^{n})^{2} + ( \delta_{tt} e^{n+1})^{2})\\
  \leq & \dfrac{1}{4 \varepsilon}(f'(\xi^{n+1})\delta_{tt}e^{n+1} ,
  \delta_t e^{n+1} + \delta_t e^n) 
  +\dfrac{L}{4\varepsilon}\|\delta_{tt}e^{n+1}\|^{2}\\
  \leq & \dfrac{L^2}{\varepsilon^2\eta}\|\delta_{tt}e^{n+1}\|_{-1}^2 +
  \dfrac{\eta}{64}\|\nabla (e^{n+1}-e^{n-1}) \|^2
  +\dfrac{L}{4\varepsilon}\|\delta_{tt}e^{n+1}\|^{2}\\
   \leq & \dfrac{L^2}{\varepsilon^2\eta}\|\delta_{tt}e^{n+1}\|_{-1}^2 
   +\dfrac{\eta}{8} \|\nabla \frac{e^{n+1}+e^n}{2} \|^2
   +\dfrac{\eta}{8} \|\nabla \frac{e^{n}+e^{n-1}}{2} \|^2
   +\dfrac{L}{4\varepsilon}\|\delta_{tt}e^{n+1}\|^{2}.\\
   \end{split}
\end{equation}
Combination of (\ref{t5}) and (\ref{1}) yields
\begin{equation}\label{t5-0}
 \begin{split}
    \widetilde{J}_{5}
    \leq & 
    \frac{L^{2}}{4\varepsilon^{2}\tilde{\eta}}\|\nabla R_{3}^{n+1}\|^{2}
    +\frac{L^{2}}{\varepsilon^{2}\tilde{\eta}}\|\nabla R_{4}^{n+1}\|^{2}
    +\left(\frac{\tilde{\eta}}{2}+\frac{L^{2}}{\varepsilon^{2}\eta} \right)\|\delta_{t}e^{n+1}\|^{2}_{-1}
    +\dfrac{L^2}{\varepsilon^2\eta}\|\delta_{tt}e^{n+1}\|_{-1}^2 \\
    &+\dfrac{3\eta}{8} \|\nabla \frac{e^{n+1}+e^n}{2} \|^2
     +\dfrac{\eta}{8} \|\nabla \frac{e^{n}+e^{n-1}}{2} \|^2
     +\dfrac{L}{4\varepsilon}\|\delta_{tt}e^{n+1}\|^{2}.\\
   \end{split}
 \end{equation}
Substituting $\widetilde{J}_{1},\cdots, \widetilde{J}_{5}$ into (\ref{t0}), we have
\begin{equation}\label{t6}
  \begin{split}
 &\frac{1}{\tau}\|\delta_{t}e^{n+1}\|_{-1}^{2}
 +\frac{\varepsilon}{2}(\|\nabla e^{n+1}\|^{2}-\|\nabla e^{n}\|^{2})
 +A\tau\|\nabla \delta_{t}e^{n+1}\|^{2}\\
 &+\frac{B}{2}(\|\delta_{t}e^{n+1}\|^{2}-\|\delta_{t}e^{n}\|^{2}+\|\delta_{tt}e^{n+1}\|^{2})\\
 \leq& \frac{1}{\tilde{\eta}}\|R_{1}^{n+1}\|_{-1}^{2}
   + \frac{A^{2}}{\tilde{\eta}}\|\nabla \Delta R_{2}^{n+1}\|^{2}
   + \Big{(}\frac{B^{2}}{\tilde{\eta}}+\frac{L^{2}}{4 \varepsilon^{2}\tilde{\eta}}\Big{)}\|\nabla R_{3}^{n+1}\|^{2}
   + \frac{\varepsilon^{2}}{\tilde{\eta}}\|\nabla \Delta R_{4}^{n+1}\|^{2}\\
   &+\frac{L^{2}}{ \varepsilon^{2}\tilde{\eta}}\|\nabla R_{4}^{n+1}\|^{2}
  +\Big{(}\frac{L^{2}}{ \varepsilon^{2}\eta }+\frac{3\tilde{\eta}}{2}\Big{)}\|\delta_{t}e^{n+1}\|^{2}_{-1}
      +\dfrac{L^2}{\varepsilon^2\eta}\|\delta_{tt}e^{n+1}\|_{-1}^2 \\
      &+\dfrac{3\eta}{8} \|\nabla \frac{e^{n+1}+e^n}{2} \|^2
      +\dfrac{\eta}{8} \|\nabla \frac{e^{n}+e^{n-1}}{2} \|^2
      +\dfrac{L}{4\varepsilon}\|\delta_{tt}e^{n+1}\|^{2}.\\     
 \end{split}
 \end{equation}
 Combining (\ref{j4}) and (\ref{t6}), then using triangle
 inequality
 $\|\delta_{tt}e^{n+1}\|^{2}_{-1}\leq 2
 \|\delta_{t}e^{n+1}\|_{-1}^{2}+2
 \|\delta_{t}e^{n}\|_{-1}^{2}$,
 \eqref{e6}-\eqref{e7-2} and following estimates
\begin{equation} \label{e6j}
\|R_{1}^{n+1}\|_{-1}^{2}\lesssim
\tau^{3}\int_{t^{n}}^{t^{n+1}}\|\partial_{tt}\phi(t)\|_{-1}^{2}{\rm d}t,
\end{equation}
\begin{equation} \label{e8j}
\|\nabla \Delta R_{2}^{n+1}\|^{2}\lesssim
\tau^{3}\int_{t^{n}}^{t^{n+1}}\|\partial_{t}\nabla \Delta \phi(t)\|^{2}{\rm d}t,
\end{equation}
\begin{equation}\label{e7j}
\|\nabla R_{3}^{n+1}\|^{2}\lesssim 6\tau^{3}\int_{t^{n-1}}^{t^{n+1}}\|\partial_{tt}\nabla \phi(t)\|^{2}{\rm d}t,
\end{equation}
\begin{equation}\label{e7-1j}
\|\nabla \Delta  R_{4}^{n+1}\|^{2}\lesssim
\tau^{3}\int_{t^{n}}^{t^{n+1}}\|\partial_{tt}\nabla \Delta \phi(t)\|^{2}{\rm d}t,
\end{equation}
\begin{equation}\label{e7-2j}
\| R_{4}^{n+1}\|_{-1}^{2}\lesssim
\tau^{3}\int_{t_{n}}^{t_{n+1}}\|\partial_{tt}\phi(t)\|_{-1}^{2}{\rm d}t,
\end{equation}
we obtain
\begin{equation}\label{t7-2}
  \begin{split}
   &\frac{1}{2\tau }(\|e^{n+1}\|_{-1}^{2}-\|e^{n}\|_{-1}^{2})
   +\varepsilon\|\nabla\frac{e^{n+1}+e^{n}}{2}\|^{2}
   +\frac{A\tau+\varepsilon}{2}(\|\nabla e^{n+1}\|^{2}-\|\nabla e^{n}\|^{2})\\
    &+\frac{1}{\tau}\|\delta_{t}e^{n+1}\|_{-1}^{2}
    +A\tau\|\nabla \delta_{t}e^{n+1}\|^{2}
    +\frac{B}{2}(\|\delta_{t}e^{n+1}\|^{2}-\|\delta_{t}e^{n}\|^{2}+\|\delta_{tt}e^{n+1}\|^{2})\\
\lesssim &
  \widetilde{C}_{1}^{n+1}\tau^{3}
  +\frac{21\eta}{8}\|\nabla \frac{e^{n+1}+e^{n}}{2}\|^{2}
  +\dfrac{\eta}{8} \|\nabla \frac{e^{n}+e^{n-1}}{2} \|^2
  \\
  &+\Big{(}\frac{2B^{2}}{\eta}+ \frac{5L^{2}}{2\varepsilon^{2}\eta }\Big{)}\|\delta_{t}e^{n}\|^{2}_{-1}	
  +\Big{(}
  \frac{2B^{2}}{\eta}+ \frac{7L^{2}}{2\varepsilon^{2}\eta}
  +\frac{3\tilde{\eta}}{2}\Big{)}\|\delta_{t}e^{n+1}\|^{2}_{-1}
  +\dfrac{L}{4\varepsilon}\|\delta_{tt}e^{n+1}\|^{2}.\\     
  &-\frac{1}{\varepsilon}\Big{(}f'(\phi(t^{n+\frac{1}{2}})) \frac{e^{n+1}+e^{n}}{2},\frac{e^{n+1}+e^{n}}{2}\Big{)}
  +\frac{L_{2}}{\varepsilon}\|\frac{e^{n+1}+e^{n}}{2}\|^{3}_{L^{3}},  \end{split}
 \end{equation}
 where
 \begin{equation}\label{t7-20}
  \begin{split}
\widetilde{C}_{1}^{n+1} =&
 \int_{t^{n}}^{t^{n+1}}\Big{(}\frac{1}{\eta}\|\partial_{tt}\Delta^{-1}\phi(t)\|_{-1}^{2}
+\frac{A^{2}}{\eta}\|\partial_{t}\nabla \phi(t)\|^{2}
+ \Big{(}\frac{\varepsilon^{2}}{\eta} + \frac{L^{2}}{ \varepsilon^{2}\tilde{\eta}} \Big{)} \|\partial_{tt}\nabla \phi(t)\|^{2}\\
&
+\Big{(}\frac{C_{2}}{\varepsilon^{2}\eta} + \frac{1}{\tilde{\eta}}\Big{)} \|\partial_{tt}\phi(t)\|_{-1}^{2}
+\frac{A^{2} }{\tilde{\eta}}\|\partial_{t}\nabla\Delta \phi(t)\|^{2}
+ \frac{\varepsilon^{2} }{\tilde{\eta}}\|\partial_{tt}\nabla \Delta \phi(t)\|^{2} \Big{)}{\rm d}t\\
&
+\int_{t^{n-1}}^{t^{n+1}}\Big{(}6\Big{(}
\frac{B^{2}}{\eta}+\frac{L^{2}}{4\varepsilon^{2}\eta}\Big{)}\|\partial_{tt}\phi(t)\|_{-1}^{2}
+ 6 \Big{(}\frac{B^{2} }{\tilde{\eta}}+\frac{L^{2}}{4 \varepsilon^{2}\tilde{\eta} }\Big{)}\|\partial_{tt}\nabla \phi(t)\|^{2}
\Big{)}{\rm d}t.
  \end{split}
\end{equation}

(iii) We now estimate the last two terms of the right hand side of (\ref{t7-2}).
The spectrum estimate (\ref{eigen}) leads to
\begin{equation}\label{spectrum}
  \varepsilon \|\nabla\frac{e^{n+1}+e^{n}}{2}\|^{2}_{L^{2}}
  +\frac{1}{\varepsilon}\Big{(}f'(\phi(t^{n+\frac{1}{2}}))\frac{e^{n+1}+e^{n}}{2}, \frac{e^{n+1}+e^{n}}{2}\Big{)}
  \geq -C_{0}\|\frac{e^{n+1}+e^{n}}{2}\|_{-1}^{2}.
\end{equation}
Applying (\ref{spectrum}) with a scaling factor $(1-\eta_{1})$ close to but smaller than 1, we get
\begin{equation}\label{spectrum1}
\begin{split}
  &-(1-\eta_{1})\frac{1}{\varepsilon}\Big{(}f'(\phi(t^{n+1}))\frac{e^{n+1}+e^{n}}{2},  \frac{e^{n+1}+e^{n}}{2}\Big{)}\\
  \leq& C_{0}(1-\eta_{1})\|\frac{e^{n+1}+e^{n}}{2}\|_{-1}^{2}+(1-\eta_{1})\varepsilon\|\nabla \frac{e^{n+1}+e^{n}}{2}\|^{2}.
  \end{split}
\end{equation}
On the other hand,
\begin{equation}\label{t9}
  -\frac{\eta_{1}}{\varepsilon}\Big{(}f'(\phi(t^{n+1}))\frac{e^{n+1}+e^{n}}{2},  \frac{e^{n+1}+e^{n}}{2}\Big{)}
  \leq \frac{L^{2}}{\varepsilon^{2}}\frac{\eta_{1}}{\eta_{2}}\|\frac{e^{n+1}+e^{n}}{2}\|_{-1}^{2}
  +\frac{\eta_{1}\eta_{2}}{4}\|\nabla\frac{e^{n+1}+e^{n}}{2}\|^{2}.
\end{equation}
Now, we estimate the $L^{3}$ term. By interpolating $L^{3}$ between $L^{2}$ and $H^{1}$ then using Poincare inequality for the error function, we get
\[\|\frac{e^{n+1}+e^{n}}{2}\|_{L^{3}}^{3}\leq K \|\nabla \frac{e^{n+1}+e^{n}}{2}\|^{\frac{d}{2}}
\|\frac{e^{n+1}+e^{n}}{2}\|^{\frac{6-d}{2}},\]
where K is a constant independent of $\varepsilon$ and $\tau$.
We continue the estimate by using
$\|\frac{e^{n+1}+ e^{n}}{2}\|^{2} \leq \|\nabla \frac{e^{n+1}+ e^{n}}{2}\| \|\frac{e^{n+1}+ e^{n}}{2}\|_{-1}$ to get
\begin{equation}\label{t20}
  \frac{L_{2}}{\varepsilon}\|\frac{e^{n+1}+e^{n}}{2}\|_{L^{3}}^{3}
\leq \frac{L_{2}}{\varepsilon}K \|\nabla \frac{e^{n+1}+e^{n}}{2}\|^{\frac{d}{2}+\frac{6-d}{4}}
\|\frac{e^{n+1}+e^{n}}{2}\|_{-1}^{\frac{6-d}{4}}
=G^{n+1}\|\nabla \frac{e^{n+1}+e^{n}}{2}\|^{2},
\end{equation}
where $G^{n+1}=\frac{L_{2}}{\varepsilon}K \|\nabla \frac{e^{n+1}+e^{n}}{2}\|^{\frac{d-2}{4}}
\|\frac{e^{n+1}+e^{n}}{2}\|_{-1}^{\frac{6-d}{4}}$.\\

\indent Now plugging equation (\ref{spectrum1}), (\ref{t9}) and (\ref{t20}) into (\ref{t7-2}), we get
   \begin{equation}\label{t11}
  \begin{split}
   &\frac{1}{2\tau }(\|e^{n+1}\|_{-1}^{2}-\|e^{n}\|_{-1}^{2})
   +\varepsilon\|\nabla\frac{e^{n+1}+e^{n}}{2}\|^{2}
   +\frac{A\tau+\varepsilon}{2}(\|\nabla e^{n+1}\|^{2}-\|\nabla e^{n}\|^{2})\\
   &+\frac{1}{\tau}\|\delta_{t}e^{n+1}\|_{-1}^{2}
   +A\tau\|\nabla \delta_{t}e^{n+1}\|^{2}
   +\frac{B}{2}(\|\delta_{t}e^{n+1}\|^{2}-\|\delta_{t}e^{n}\|^{2}+\|\delta_{tt}e^{n+1}\|^{2})\\
\lesssim  & \widetilde{C}_{1}^{n+1}\tau^{3}
+\Big{(}\frac{21\eta}{8}+(1-\eta_{1})\varepsilon+\frac{\eta_{1}\eta_{2}}{4}\Big{)}\|\nabla \frac{e^{n+1}+e^{n}}{2}\|^{2}
+\dfrac{\eta}{8} \|\nabla \frac{e^{n}+e^{n-1}}{2} \|^2
\\
&+\Big{(}\frac{2B^{2}}{\eta}+ \frac{5L^{2}}{2\varepsilon^{2}\eta }\Big{)}\|\delta_{t}e^{n}\|^{2}_{-1}	
+\Big{(}
\frac{2B^{2}}{\eta}+ \frac{7L^{2}}{2\varepsilon^{2}\eta}
+\frac{3\tilde{\eta}}{2}\Big{)}\|\delta_{t}e^{n+1}\|^{2}_{-1}
+\dfrac{L}{4\varepsilon}\|\delta_{tt}e^{n+1}\|^{2}.\\     
  &+\Big{(}C_{0}(1-\eta_{1})+\frac{L^{2}}{\varepsilon^{2}}\frac{\eta_{1}}{\eta_{2}}\Big{)}
  \|\frac{e^{n+1}+e^{n}}{2}\|_{-1}^{2}
   +G^{n+1}\|\nabla \frac{e^{n+1}+e^{n}}{2}\|^{2}.
   \end{split}
 \end{equation}
 Take $\eta_{1}=\varepsilon^{3}, \ \eta_{2}=\varepsilon,\ 
\eta=\varepsilon^{4}/{11}$, $\tilde{\eta}=\varepsilon^{-6}$,
such that
\[ \frac{L^{2}}{\varepsilon^{2}}\frac{\eta_{1}}{\eta_{2}}=L^{2}, \ \
\frac{22\eta}{8} + (1-\eta_{1})\varepsilon+\frac{\eta_{1}\eta_{2}}{4}
= \varepsilon-\frac{\varepsilon^{4}}{2},\]
and
\begin{equation}\label{eq:es65}
\Big{(}\frac{22\eta}{8}+(1-\eta_{1})\varepsilon+\frac{\eta_{1}\eta_{2}}{4}\Big{)}\|\nabla \frac{e^{n+1}+e^{n}}{2}\|^{2}
= (\varepsilon-\frac{\varepsilon^4}{2}) 
\|\nabla \frac{e^{n+1}+e^{n}}{2}\|^{2}.
\end{equation}
Take 
\begin{equation}
\tau \leq 
\frac{1}{2\Big(\frac{4B^{2}}{\eta}+\frac{6L^2}{\varepsilon^2\eta}
+\frac{3\tilde{\eta}}{2}\Big) } 
\lesssim \varepsilon^6,
\end{equation}
such that
\begin{equation} \label{eq:es66}
\Big(\frac{4B^{2}}{\eta}+\frac{6L^2}{\varepsilon^2\eta}
+\frac{3\tilde{\eta}}{2}\Big)\|\delta_{t}e^{n+1}\|^{2}_{-1}
\leq
\frac{1}{2\tau} \|\delta_{t}e^{n+1}\|^{2}_{-1}.
\end{equation}
Summing up (\ref{t11}), \eqref{eq:es65} and \eqref{eq:es66}, we get
\begin{equation}\label{t12}
   \begin{split}
   &\frac{1}{2\tau }(\|e^{n+1}\|_{-1}^{2}-\|e^{n}\|_{-1}^{2})
   +\frac{\varepsilon^4}{2}\|\nabla\frac{e^{n+1}+e^{n}}{2}\|^{2}
   +\frac{A\tau+\varepsilon}{2}(\|\nabla e^{n+1}\|^{2}-\|\nabla e^{n}\|^{2})\\
   &+\frac{1}{2\tau}\|\delta_{t}e^{n+1}\|_{-1}^{2}
   +A\tau\|\nabla \delta_{t}e^{n+1}\|^{2}
   +\frac{B}{2}(\|\delta_{t}e^{n+1}\|^{2}-\|\delta_{t}e^{n}\|^{2})
   +\left(\frac{B}{2}-\frac{L}{4\varepsilon}\right)\|\delta_{tt}e^{n+1}\|^{2}\\
   &+\dfrac{\varepsilon^4}{88}\left(\|\nabla\frac{e^{n+1}+e^{n}}{2}\|^2-\|\nabla \frac{e^{n}+e^{n-1}}{2}\|^2 \right)
   +11\Big(\frac{2B^{2}}{\varepsilon^4}+ \frac{5L^{2}}{2\varepsilon^{6}}\Big)\left(\|\delta_{t}e^{n+1}\|^{2}_{-1}-\|\delta_{t}e^{n}\|^{2}_{-1} \right)\\
\lesssim  & \widetilde{C}_{1}^{n+1}\tau^{3}
   +\Big{(}C_{0}(1-\varepsilon^3)+L^2\Big{)}
   \left(\frac12\|e^{n+1}\|_{-1}^{2}+\frac12\|e^{n}\|_{-1}^{2} \right)
   +G^{n+1}\|\nabla \frac{e^{n+1}+e^{n}}{2}\|^{2}.
   \end{split}
 \end{equation}
 Now, if $G^{n+1}$ is uniformly bounded by constant
 $\varepsilon^{4}/4$, we can multiply by $2\tau$ on both
 sides of inequality (\ref{t12}), and sum up for $n=1$ to
 $N$ to get the following estimate:
\begin{equation}\label{t13}
\begin{split}
&\|e^{N+1}\|_{-1}^{2}  
+\tau(A\tau+\varepsilon)\|\nabla e^{N+1}\|^{2}
+B\tau\|\delta_{t}e^{N+1}\|^{2}
+\dfrac{\tau\varepsilon^4}{44}\|\nabla\frac{e^{N+1}+e^{N}}{2}\|^2
\\
&
+22\tau\Big(\frac{2B^{2}}{\varepsilon^4}+\frac{5L^{2}}{2\varepsilon^{6}}\Big)\|\delta_{t}e^{N+1}\|^{2}_{-1}\\
&+\sum_{n=1}^{N}\left[\frac{\tau\varepsilon^4}{2}\|\nabla\frac{e^{n+1}+e^{n}}{2}\|^{2}
+\|\delta_{t}e^{n+1}\|_{-1}^{2}
+2A\tau^2\|\nabla \delta_{t}e^{n+1}\|^{2}
+\tau\big(B-\frac{L}{2\varepsilon}\big)\|\delta_{tt}e^{n+1}\|^{2}
\right]\\
\lesssim  & 2\widetilde{C}_{1}\tau^{4}
+\left(1+22\tau\Big(\frac{2B^{2}}{\varepsilon^4}+\frac{5L^{2}}{2\varepsilon^{6}}\Big)\right)\|e^{1}\|^{2}_{-1}
+B\tau\|e^{1}\|^{2}
+\left(\tau(A\tau+\varepsilon)+\dfrac{\tau\varepsilon^4}{176}\right)\|\nabla{e^{1}}\|^2
\\
&+\tau\Big{(}C_{0}+L^2\Big{)}\|e^{N+1}\|_{-1}^{2}
+2\tau\Big{(}C_{0}+L^2\Big{)}
\sum_{n=1}^N\|e^{n}\|_{-1}^{2},
\end{split}
\end{equation}
where
 \begin{equation}\label{t13-0}
 \begin{split}
   \widetilde{C}_{1}=&\sum_{n=0}^{n=N}\widetilde{C}_{1}^{n+1}\\
    \leq &
    \int_{0}^{T}\Big{(}\frac{1}{\eta}\|\partial_{tt}\Delta^{-1}\phi(t)\|_{-1}^{2}
    +\frac{A^{2}}{\eta}\|\partial_{t}\nabla \phi(t)\|^{2}
    + \Big{(}\frac{\varepsilon^{2}}{\eta} + \frac{L^{2}}{ \varepsilon^{2}\tilde{\eta}} \Big{)} \|\partial_{tt}\nabla \phi(t)\|^{2}\\
    &
    +\Big{(}\frac{C_{2}}{\varepsilon^{2}\eta} + \frac{1}{\tilde{\eta}}\Big{)} \|\partial_{tt}\phi(t)\|_{-1}^{2}
    +\frac{A^{2} }{\tilde{\eta}}\|\partial_{t}\nabla\Delta \phi(t)\|^{2}
    + \frac{\varepsilon^{2} }{\tilde{\eta}}\|\partial_{tt}\nabla \Delta \phi(t)\|^{2}\\
    &
    +12\Big{(}
    \frac{B^{2}}{\eta}+\frac{L^{2}}{4\varepsilon^{2}\eta}\Big{)}\|\partial_{tt}\phi(t)\|_{-1}^{2}
    + 12\Big{(}\frac{B^{2} }{\tilde{\eta}}+\frac{L^{2}}{4 \varepsilon^{2}\tilde{\eta} }\Big{)}\|\partial_{tt}\nabla \phi(t)\|^{2}
    \Big{)}{\rm d}t.\\
\lesssim & \varepsilon^{-\max\{\rho_{1}+4, \rho_{2}+6, \rho_{4}+2, \rho_{5}-8, \rho_{6}+8, \rho_{7}-2\}}.\\
 \end{split}
 \end{equation}
 Choose $\tau \leq 1/{(2C_{0}+2L^{2})}$, then we can get a
 finer error estimate by discrete Gronwall inequality
 and the assumption of first step error (\ref{ap:022-2}):
  \begin{equation}\label{t14}
  \begin{split}
   &\max_{1\leq n\leq N}\Big(
   \|e^{n+1}\|_{-1}^{2}
   +\tau(A\tau+\varepsilon)\|\nabla e^{n+1}\|^{2}
   +B\tau\|\delta_{t}e^{n+1}\|^{2}
   +\tau\varepsilon^{-6}\|\delta_{t}e^{n+1}\|^{2}_{-1}\Big)
   \\
   &+\sum_{n=1}^{N}\left[\frac{\tau\varepsilon^4}{2}\|\nabla\frac{e^{n+1}+e^{n}}{2}\|^{2}
   +\|\delta_{t}e^{n+1}\|_{-1}^{2}
   +2A\tau^2\|\nabla \delta_{t}e^{n+1}\|^{2}
   +\tau\big(B-\frac{L}{2\varepsilon}\big)\|\delta_{tt}e^{n+1}\|^{2}
   \right]
	\\
  \lesssim &\varepsilon^{-\sigma}
  \mathrm{exp }(4(C_{0}+L^{2}+1)T)\tau^{4}.\\
 \end{split}
 \end{equation}
  We prove this by induction. Assuming that
 the above estimate holds for all first $N$ time steps.
 Since $\tau \lesssim \varepsilon^{6}$,
 then the coarse estimate \eqref{cet2-3} leads to
 \begin{equation}\label{cet2-3-0}
  \begin{split}
   &\|e^{N+1}\|_{-1}^{2}
   +\frac{\varepsilon \tau}{4}\|\nabla \frac{e^{N+1}+e^{N}}{2}\|^{2}
   +A\tau^{2}\|\nabla e^{N+1}\|^{2}
   +B\tau\|\delta_{t}e^{N+1}\|^{2}\\
  \lesssim &
 \varepsilon^{-\max\{\rho_{1}+1, \rho_{2}+3, \rho_{4}-1, \rho_{6}+5\}} \tau^{4}
  +\varepsilon^{-\sigma}
  \mathrm{exp }(4(C_{0}+L^{2})T)\tau^{4}\\
  \lesssim &  \varepsilon^{-\sigma} \tau^{4}.
   \end{split}
 \end{equation}
  To obtain $G^{N+1}\leq \varepsilon^{4}/4$, using (\ref{cet2-3-0}), we easily get
  \begin{equation}\label{t21}
  \begin{split}
  G^{N+1}=&\frac{L_{2}}{\varepsilon}K \|\nabla \frac{e^{N+1}+e^{N}}{2}\|^{\frac{d-2}{4}}
\|\frac{e^{N+1}+e^{N}}{2}\|_{-1}^{\frac{6-d}{4}}\\
\leq &\frac{L_{2}}{\varepsilon}K' \Big{(}
 \varepsilon^{-\sigma-1}
   \tau^{3}\Big{)}^{\frac{d-2}{8}} 
   \Big{(}
    \varepsilon^{-\sigma}
   \tau^{4}\Big{)}^{\frac{6-d}{8}}
   \leq \dfrac{\varepsilon^{4}}{4}.
  \end{split}
  \end{equation}
 Solving (\ref{t21}), we get
 \begin{equation}\label{t16}
 \begin{split}
   \tau
  \lesssim \varepsilon^{\frac{1}{18-d}(4\sigma+d +38)}.\\
  \end{split}
\end{equation}
The proof is complete.
\end{proof}

\begin{remark}\label{rem:conv}
  Note that the spectral estimate \eqref{eigen} is essential
  to the proof. Moreover, since the Crank-Nicolson
  discretization has no numerical diffusion, it is harder to
  bound the error growth than the BDF2 scheme. Here, we need
  $B>\frac{L}{2\varepsilon}$ to get the convergence, while
  in SL-BDF2 scheme, there is no such a requirement
  \cite{wang_two_2017}.
\end{remark}

\begin{remark}\label{rem:conv2}
	We used $L^\infty$ bound assumption of the exact solution to
	handle the high order  term $\Big((R_4^{n+1})^2, \frac{e^{n+1}+e^n}{2}\Big)$ occured in \eqref{j3}. 
	There is another way to control this term. By Cachy-Schwartz inequality, one only need to control
	$\|R_4^{n+1}\|^4_{L^4}$ and $\|\frac{e^{n+1}+e^n}{2}\|^2$.
	The $L^4$ term can be controlled by using Sobolev interpolation inequality as we did for the $\|\frac{e^{n+1}+e^n}{2}\|_{L^3}^3$ term. The $L^2$ term of the error function can be controlled by a $\frac{\varepsilon^4}{8}\|\nabla\frac{e^{n+1}+e^n}{2}\|^2$ term and $\frac{1}{\tau}\|\frac{e^{n+1}+e^n}{2}\|_{-1}^2$.
\end{remark}

\section{Numerical results}

In this section, we numerically verify our
schemes are energy
stable and second order accurate in time.

We use the commonly used  double-well potential 
$F(\phi)=\frac{1}{4}(\phi^{2}-1)^{2}$.
It is a common practice to modify $F(\phi)$ to have a
quadratic growth rate for $|\phi|>1$ (since physically
$|\phi|\leq 1$), such that a global Lipschitz condition is
satisfied \cite{shen_numerical_2010},
\cite{condette_spectral_2011}. To get a $C^4$ smooth
double-well potential with quadratic growth, we introduce
$\tilde{F}(\phi) \in C^{\infty}(\mathbf{R})$ as a smooth
mollification of
\begin{equation}\label{efnew}
\hat{F}(\phi)=
\begin{cases}
\frac{11}{2}(\phi-2)^{2}+6(\phi-2)+\frac94, &\phi>2, \\
\frac{1}{4 }(\phi^{2}-1)^{2}, &\phi\in [-2,2], \\
\frac{11}{2}(\phi+2)^{2}+6(\phi+2)+\frac94, &\phi<-2.
\end{cases}
\end{equation}
with a mollification parameter much smaller than 1, to
replace $F(\phi)$. Note that the truncation points $-2$ and
$2$ used here are for convenience only. Other values outside
of region $[-1,1]$ can be used as well.  For simplicity, we
still denote the modified function $\tilde{F}$ by $F$.

To test the numerical scheme, we solve \eqref{eq:CH} in
tensor product 2-dimensional domain
$\Omega=[-1,1]\times[-1,1]$. We use a Legendre Galerkin
method similar as in
\cite{shen_efficient_2015, yang_yu_efficient_2017} for spatial
discretization.  Let $L_k(x)$ denote the Legendre polynomial
of degree $k$. We define
\[
V_M = \mbox{span}\{\,\varphi_k(x)\varphi_j(y),\ k,j=0,\ldots, M-1\,\}\in H^1(\Omega),
\]
where
$\varphi_0(x) = L_0(x); \varphi_1(x)=L_1(x);  \varphi_k(x)=L_k(x)-L_{k+2}(x), k=2,\ldots, M\!-\!1
$,
be the Galerkin approximation space for both $\phi^{n+1}$ and $\mu^{n+1}$. Then the full discretized
form for the SL-CN scheme reads:

Find $(\phi^{n+1}, \mu^{n+\frac12}) \in (V_M)^{2}$ such that
\begin{equation}\label{eq:bdf2:fulldis1}
\frac{1}{\tau}(\phi^{n+1}-\phi^{n}, \omega)=-\gamma(\nabla \mu^{n+\frac12},\nabla \omega), 
\quad \forall\, \omega \in V_M,
\end{equation}
\begin{equation}\label{eq:bdf2:fulldis2}
\begin{split}
(\mu^{n+\frac12},\varphi)={} &\frac{\varepsilon}{2} (\nabla (\phi^{n+1}+\phi^n),\nabla \varphi) +
\frac{1}{\varepsilon}(f(\frac32\phi^{n}-\frac12\phi^{n-1}),\varphi)
\\
&\quad+A\tau (\nabla \delta_{t}\phi^{n+1},\nabla \varphi)
+B(\delta_{tt}\phi^{n+1},\varphi), \quad \forall\, \varphi \in
V_M.
\end{split}
\end{equation}
This is a linear system with constant coefficients for
$(\phi^{n+1}, \mu^{n+\frac12})$, which can be efficiently solved.
We use a spectral transform with doubled quadrature points to
eliminate the aliasing error and efficiently evaluate the
integration $(f(\frac32\phi^n-\frac12\phi^{n-1}),\varphi)$ in equation
\eqref{eq:bdf2:fulldis2}.

We take $\varepsilon=0.05$ and $M=63$ and use two different
initial values to test the stability and accuracy of the
proposed schemes:
\begin{enumerate}
	\item[(1)] $\{\phi_0(x_i,y_j) \}\in\, {\bf{R}}^{2M\times2M}$
	with $x_i, y_j$ are tensor product Legendre-Gauss
	quadrature points and $\phi_0(x_i,y_j)$ is a uniformly
	distributed random number between $-1$ and $1$ (shown in
	the left picture of Fig. \ref{fig1});
	\item[(2)] The solution of the Cahn-Hilliard equation at
	$t=64\varepsilon^3$ which takes $\phi_0$ as its initial
	value (Denoted by $\phi_1$ shown in the middle picture of
	Fig. \ref{fig1}).
\end{enumerate}
\begin{figure}[htbp]
	\centering
	\includegraphics[width=\textwidth]{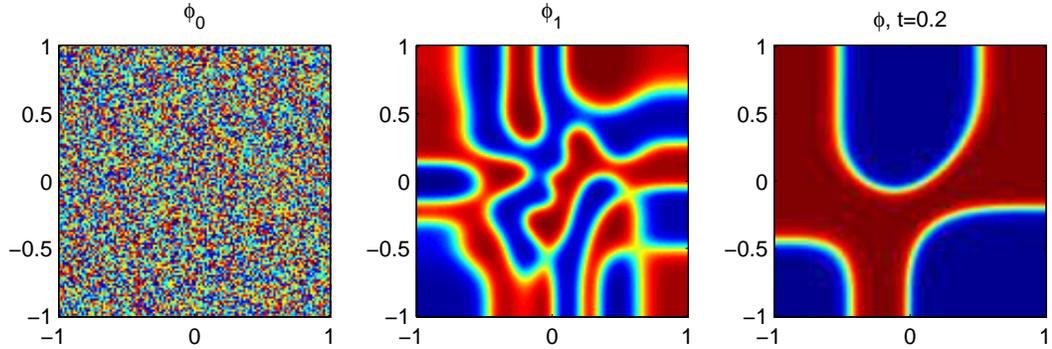}
	\caption{The two random initial values $\phi_0$,
		$\phi_1$ and the state of $\phi_1$ evolves $0.2$ time
		unit according to the Cahn-Hilliard equation
		\eqref{eq:CH} with $\gamma=1$.}
	\label{fig1}
\end{figure}

\subsection{Stability results}

Table \ref{tstab0:CHCNLL} shows the required minimum values
of $A$ (resp. $B$) with different $\gamma$, $B$ (resp. $A$)
and $\tau$ values for stably solving (not blow up in 4096
time steps) the Cahn-Hilliard equation \eqref{eq:CH} with
initial value $\phi_0$.  The results for the initial value
$\phi_1$ are similar. From this table, we observe that the
SL-CN scheme is stable with $A=0, B=0$ when $\tau$ is small
enough.  If we take $A=0$, then $B=16$ will make the scheme
unconditionally stable, the values of $\gamma$ has only a
very small effect on the values of $B$. But when we fix
$B$, the case $\gamma=1$ requires a much larger $A$ value to
make the scheme stable than $\gamma=0.0025$ case, this is
consistent to our analysis.

\begin{table}[htbp]
	\begin{tabular}{|c|c|c|c|c|c|c|c|c|}
		\hline
		\multirow{3}{*}{$\tau$ }
		&  \multicolumn{4}{|c|}{Minimum A required}
		&  \multicolumn{4}{|c|}{Minimum B required}
		\\ \cline{2-9}&  \multicolumn{2}{|c|}{$\gamma=0.0025$ }
		&  \multicolumn{2}{|c|}{$\gamma=1$ }
		&  \multicolumn{2}{|c|}{$\gamma=0.0025$ }
		&  \multicolumn{2}{|c|}{$\gamma=1$ }
		\\ \cline{2-9} & $B=0$ & $B=10$ & $B=0$ & $B=10$ & $A=0$ & $A=4$ & $A=0$ & $A=4$\\ \hline
		10  & 0.16 & 0.005 & 1 & 1 & 16 & 8 & 16 & 0\\ \hline
		1  & 0.16 & 0 & 8 & 1 & 16 & 16 & 16 & 2\\ \hline
		0.1  & 0.16 & 0 & 32 & 1 & 8 & 4 & 16 & 8\\ \hline
		0.01  & 0.08 & 0 & 64 & 2 & 8 & 4 & 16 & 8\\ \hline
		0.001  & 0 & 0 & 64 & 0 & 0 & 0 & 8 & 8\\ \hline
		0.0001  & 0 & 0 & 64 & 0 & 0 & 0 & 8 & 4\\ \hline
		1E-05  & 0 & 0 & 32 & 0 & 0 & 0 & 2 & 2\\ \hline
		1E-06  & 0 & 0 & 0 & 0 & 0 & 0 & 0 & 0\\ \hline
	\end{tabular} 
	\centering 
	\caption{The minimum values of $A$(resp $B$) (only values $\{0, 2^i,i=0,\ldots,7\}\times\gamma$ are tested for $A$, only values $\{0, 2^i,i=0,\ldots,7\}$ are tested for $B$) 
      to make scheme SL-CN stable when $\gamma$, $B$ (resp $A$) and $\tau$ taking different values.
	}
	\label{tstab0:CHCNLL}
	\end{table}

Figure \ref{fig:2Energy} presents the discrete energy
dissipation of the SL-CN scheme using several
time step-sizes. We see clearly the energy decaying property is
maintained. Moreover, as $t$ increases, the differences between $E$ and $E_{CN}$ get smaller and smaller.

\begin{figure}[htbp]
	\centering 
	\includegraphics[width=0.6\textwidth]{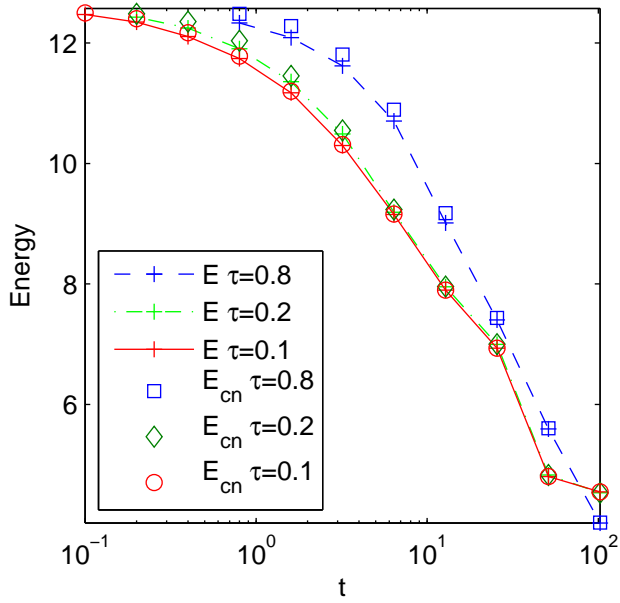}
	\caption{The discrete energy dissipation of the SL-CN
		scheme solving the Cahn-Hilliard equation with
		initial value $\phi_1$, and relaxation parameter
		$\gamma=0.0025$. Stability constant $A=1, B=20$ are
		used. }
	\label{fig:2Energy}
\end{figure}

\subsection{Accuracy results}

We take initial value $\phi_1$ to test the accuracy of the
two schemes.  The Cahn-Hilliard equation with
$\gamma=0.0025$ are solved from $t=0$ to $T=12.8$. To
calculate the numerical error, we use the numerical result
generated using $\tau=10^{-3}$ as a reference of exact
solution.  The results are given in Table
 \ref{tbl:conv:SL-CN}. We
see that the scheme is second order accuracy in
$H^{-1}, L^2$ and $H^1$ norm.  

\begin{table}[htbp]
	\begin{tabular}{|c|c|c|c|c|c|c|}
		\hline
		$\tau$ & $H^{-1}$ Error & Order & $L^2$ Error & Order & $H^1$ Error & Order\\ \hline
		0.16 & 7.98E-02 &  & 5.20E-01 &  & 6.40E+00 &  \\ \hline
		0.08 & 2.18E-02 & 1.87 & 1.64E-01 & 1.66 & 2.18E+00 & 1.56 \\ \hline
		0.04 & 5.95E-03 & 1.87 & 4.57E-02 & 1.85 & 6.08E-01 & 1.84 \\ \hline
		0.02 & 1.54E-03 & 1.95 & 1.16E-02 & 1.97 & 1.55E-01 & 1.97 \\ \hline
		0.01 & 3.86E-04 & 2.00 & 2.90E-03 & 2.00 & 3.87E-02 & 2.00 \\ \hline
		0.005 & 9.38E-05 & 2.04 & 7.05E-04 & 2.04 & 9.39E-03 & 2.04 \\ \hline
	\end{tabular} 
	\centering 
	\caption{The convergence of the SL-CN scheme
		with $B=40$, $A=0.1$
		for the Cahn-Hilliard equation with initial value $\phi_1$,parameter $\gamma=0.0025$.
		The errors are calculated at $T=12.8$.}
	\label{tbl:conv:SL-CN}
\end{table}

\section{Conclusions}

We study the stability and convergence of a stabilized linear
Crank-Nicolson scheme for the Cahn-Hilliard phase field
equation.  The scheme includes two second-order
stabilization terms, which guarantee the unconditional
energy dissipation theoretically. Use a standard error
analysis procedure for parabolic equation, we get an error
estimate with a prefactor depending on $1/\varepsilon$
exponentially. We then refine the result by using a spectrum
estimate of the linearized Cahn-Hilliard operator and
mathematical induction to get an optimal (second-order) convergence
estimate in $l^{\infty}(0,T;H^{-1})\cap l^{2}(0,T;H^{1})$
norm with a prefactor depends only on some lower degree
polynomial of $1/\varepsilon$.  Numerical results are
presented to verify the stability and accuracy of the scheme.

\section*{Acknowledgment}
This work is partially supported by Major Program of NNSFC
under Grant 91530322 and NNSFC Grant 11371358,11771439.
The authors thank Prof. Jie Shen and Prof. Xiaobing Feng for helpful discussions.

\section*{References}
\bibliographystyle{alpha}  
\newcommand{\etalchar}[1]{$^{#1}$}

\end{document}